\documentclass[12pt,reqno]{amsart}
\usepackage{amsmath}
\usepackage{amssymb}
\usepackage{amstext}
\usepackage{mathrsfs}
\usepackage{a4wide}
\usepackage{graphicx}
\usepackage{bbm}
\allowdisplaybreaks \numberwithin{equation}{section}
\usepackage{color}
\usepackage{cases}

\numberwithin{equation}{section}

\newtheorem{theorem}{Theorem}[section]

\newtheorem{lemma}[theorem]{Lemma}

\theoremstyle{definition}

\theoremstyle{remark}
\newtheorem{remark}[theorem]{Remark}

\begin{document}

\title
{Global solutions for the generalized SQG equation and rearrangements}

\author{Daomin Cao, Guolin Qin,  Weicheng Zhan, Changjun Zou}

\address{Institute of Applied Mathematics, Chinese Academy of Sciences, Beijing 100190, and University of Chinese Academy of Sciences, Beijing 100049,  P.R. China}
\email{dmcao@amt.ac.cn}
\address{Institute of Applied Mathematics, Chinese Academy of Sciences, Beijing 100190, and University of Chinese Academy of Sciences, Beijing 100049,  P.R. China}
\email{qinguolin18@mails.ucas.edu.cn}
\address{Institute of Applied Mathematics, Chinese Academy of Sciences, Beijing 100190, and University of Chinese Academy of Sciences, Beijing 100049,  P.R. China}
\email{zhanweicheng16@mails.ucas.ac.cn}

\address{Institute of Applied Mathematics, Chinese Academy of Sciences, Beijing 100190, and University of Chinese Academy of Sciences, Beijing 100049,  P.R. China}
\email{zouchangjun17@mails.ucas.ac.cn}


\begin{abstract}
In this paper, we study the existence of rotating and traveling-wave solutions for the generalized surface quasi-geostrophic (gSQG) equation. The solutions are obtained by maximization of the energy over the set of rearrangements of a fixed function. The rotating solutions take the form of co-rotating vortices with $N$-fold symmetry. The traveling-wave solutions take the form of translating vortex pairs. Moreover, these solutions constitute the desingularization of co-rotating $N$ point vortices and counter-rotating pairs. Some other quantitative properties are also established.
\end{abstract}

\maketitle

\section{Introduction and main results}
In this paper, we consider the generalized surface quasi-geostrophic (gSQG) equation
\begin{align}\label{1-1}
\begin{cases}
\partial_t\vartheta+\mathbf{v}\cdot \nabla \vartheta =0&\text{in}\ \mathbb{R}^2\times (0,T)\\
 \ \mathbf{v}=\nabla^\perp(-\Delta)^{-s}\vartheta     &\text{in}\ \mathbb{R}^2\times (0,T),\\
\end{cases}
\end{align}
where ${1}/{2}\le s\le1$, $\vartheta(x,t):\mathbb{R}^2\times (0,T)\to \mathbb{R}$ is the active scaler being transported by the velocity field $\mathbf{v}(x,t):\mathbb{R}^2\times (0,T)\to \mathbb{R}^2$ generated by $\vartheta$, and $(a_1,a_2)^\perp=(a_2,-a_1)$.

When $s=1$, \eqref{1-1} is the vorticity formulation of the two-dimensional incompressible Euler equation; see \cite{MB}. When $s={1}/{2}$, \eqref{1-1} is the surface quasi-geostrophic (SQG) equation, which arises from oceanic and atmospheric science; see \cite{He, La}. \eqref{1-1} was proposed by C\'{o}rdoba et al. in \cite{Cor} as an interpolation between the Euler equation and the SQG equation. This model has a strong mathematical and physical analogy with the three-dimensional incompressible Euler equation, see \cite{Con} and the references therein.

For the two-dimensional incompressible Euler equation (i.e., $s=1$), the global well-posedness of \eqref{1-1} with initial data in $L^1\cap L^\infty$ was established by Yudovich \cite{Yud}; see also Bardos \cite{Bar} for another proof of existence. Since the work of Yudovich, the theory of weak solutions has been considerably improved. The $L^1$ assumption can be replaced by an appropriate symmetry condition thanks to the work of Elgindi and Jeong \cite{Elg}. In \cite{Dip}, DiPerna and Majda introduced solutions with $L^p$ vorticity. Delort \cite{Del} considered solutions whose vorticy is a signed measure in $H^{-1}$. We refer to \cite{MB} for extensive discussion and bibliography in this respect. Compared with the Euler equation, it seems more complicated for the case ${1}/{2}\le s<1$. It is delicate to extend the Yudovich theory of weak solutions to the gSQG equation because the velocity is, in general, not Lipschitz continuous. Local well-posedness of classical solutions of the SQG equation was established by Constantin et al. in \cite{Con}. Local well-posedness of the gSQG equation is known for sufficiently regular initial data; see \cite{Cas2, Chae, Gan, Kis1, Rod} and the references therein. Local existence was also studied in various function spaces, see for example \cite{Chae0, Li0, Wu1, Wu2}. Resnick \cite{Res} showed global existence of weak solutions of the SQG equation with any initial data in $L^2$. Marchand \cite{Mar} extended this result to the class of initial data belonging to $L^p$ with $p>4/3$ (see \cite{Chae, Gan, Nah} for more general classes of weak solutions). The non-uniqueness for weak solutions is a challenging open problem. Buckmaster et al. \cite{Buc} proved that below a certain regularity threshold, weak solutions to the SQG equation are not unique. The problem of whether the gSQG system presents finite time singularities or there is global well-posedness of classical solutions is still open (except for the Euler equation $s=1$). The early numerical simulations in \cite{Con} indicated a possible singularity in the form of a hyperbolic saddle closing in finite time. The scenario of a collapsing hyperbolic saddle blowup was ruled out by C\'{o}rboda and Fefferman \cite{Cor1, Cor2}. In \cite{Kis}, Kiselev and Nazarov constructed solutions of the SQG equation (in a periodic setting) with arbitrary Sobolev growth. For more discussion in this direction, we refer the reader to \cite{Cas2, Kis1, Kis2, Yu} and the references therein.

It is well known that all radially symmetric functions $\vartheta$ are stationary solutions to the gSQG equation due to the structure of the nonlinear term. An interesting issue is to construct other global solutions which do not change form as time evolves (called relative equilibria). The two most typical types are rotating and traveling-wave solutions, which are also the main objects focused on in the present paper. Their dynamics appear as simple flow configurations described by rotational or translational motions, but in fact they are extremely rich and exhibit complex behaviors.

In the study of rotating solutions, uniformly rotating patches (also known as V-states) seem to be the most favored. There is abundant literature devoted to the study of the V-states. The first example of explicit non trivial rotating patches for the Euler equation was discovered by Kirchhoff \cite{Kir}, which is an ellipse of semi-axes $a$ and $b$ is subject to a perpetual rotation with uniform angular velocity $ab/(a+b)^2$. In the 1970s, Deem and Zabusky \cite{Deem} gave some numerical observations of the existence of more V-states with $N$-fold symmetry for $N\ge 2$. Recall that a domain is said to be $N$-fold symmetric if it is invariant under the rotation around its center with an angle of $2\pi/N$. The Kirchhoff ellipse is clearly $2$-fold symmetric. By using bifurcation theory, Burbea \cite{Burb} provided a first rigorous proof of existence for nonlinear rotating patches for $N\ge3$ (see also \cite{Hmi}). In the spirit of Burbea's approach, Hassainia and Hmidi \cite{Has} established the existence of the V-states for the gSQG equation with $s\in (1/2,1)$. In \cite{Cas1}, Castro et al. proved the existence of the V-states for the remaining open cases $s\in(0,1/2]$. The construction of doubly connected V-states for the gSQG equation was considered in \cite{de1, de2}. In \cite{T2}, Turkington proved the existence of co-rotating vortex patches with $N$-fold symmetry for the Euler equation by using variational method. Hmidi and Mateu \cite{HM} gave a direct proof of the existence of co-rotating and counter-rotating vortex pairs of simply connected patches for the gSQG equation with $s\in (1/2,1]$ by using the contour dynamics equations. Recently, co-rotating vortex patches with $N$-fold symmetry for the gSQG equation was obtained by Garc\'ia \cite{Gar2} and Godard-Cadillac et al. \cite{Go}. Besides the V-states, exhibiting global rotating smooth solutions for the gSQG equation is also a matter of concern. In \cite{Cas3, Cas2}, Castro et al. first showed the existence of uniformly rotating smooth solutions for the gSQG equation by developing a bifurcation argument from a specific radially symmetric function. In \cite{Ao}, Ao et al. successfully applied the Lyapunov-Schmidt reduction method to construct travelling and rotating smooth solutions to the gSQG equation with $s\in(0,1)$. For other relevant topics in this field, see \cite{Ber, Cas4, Che, Dri, Gar1, Gom, Has0, Hmi, Hmi1, Hmi2} and the references therein.

In the study of travelling-wave solutions, translating vortex pairs is the main concern. The literature of vortex pairs can be traced back to the work of Pocklington \cite{Poc} in 1895. In 1906, Lamb \cite{Lamb} noted an explicit solution for the Euler equation which is now generally referred to as the Lamb dipole or Chaplygin-Lamb dipole; see \cite{Mel}. The uniqueness for the circular vortex pair was established by Burton \cite{Bu1, Bu2}. Recently, Abe and Choi \cite{Abe} proved orbital stability of the Lamb dipole. Exact travelling solutions are known only in special cases. Besides those exact solutions, the existence (and abundance) of translating vortex pairs for the Euler equation has been rigorously established. Norbury \cite{Nor} first constructed smooth translating vortex pairs by using variational method. In \cite{T1}, Turkington showed the existence of translating vortex pairs in patch setting. Burton \cite{Bu0} and Badiani \cite{Bad1} studied the existence of translating vortex pairs with a prescribed distribution of the vorticity. Inspired by Burton's work, Elcrat and Miller \cite{Elc1, Elc2} extended Turkington's result to the context of rearrangements of vorticity. By using the mountain pass lemma, the existence of translating vortex pairs was also proved by Ambrosetti and Yang \cite{Amb}, Yang \cite{Yang}, and Smets and Van Schaftingen \cite{SV}. Recently, Cao et al. \cite{Cao4} obtained a wide class of variational solutions with general profile functions. As for the gSQG equation, Hmidi and Mateu \cite{HM} obtained travelling vortex pairs in patch setting for the gSQG equation with $s\in (1/2,1]$. In \cite{Gra}, Gravejat and Smets first proved the existence of smooth translating vortex pairs for the SQG equation. Godard-Cadillac \cite{Go0} generalized this result to the gSQG equation with $s\in(0,1)$. As mentioned above, smooth travelling solutions for the gSQG equation were also obtained by Ao et al. \cite{Ao}.

In this paper, we further study the existence of rotating and traveling-wave solutions for the gSQG equation. Before stating our main result, we first explicit the equations satisfied by rotating and traveling-wave solutions. We look for rotating solutions $\vartheta$ to \eqref{1-1} under the form
\begin{equation}\label{1-2}
	\vartheta(x,t)=\omega(Q_{-\alpha t}x),
\end{equation}	
where $\omega(x)$ is some profile function defined on $\mathbb{R}^2$, and
\begin{equation*}
	Q_{\phi}=\left(
	\begin{array}{ccc}
		\cos\phi & -\sin\phi \\
		\sin\phi & ~\cos\phi \\
	\end{array}
	\right)
\end{equation*}
is the counterclockwise rotation operator of angle $\phi$. Recall that the operator $(-\Delta)^{-s}$ is given by the expression
\begin{equation*}
	(-\Delta)^{-s}\omega(x)=\mathcal{G}_s\omega(x)=\int_{\mathbb{R}^2}G_s(x-y)\omega(y)dy,
\end{equation*}
where $G_s$ is the fundamental solution of $(-\Delta)^{s}$ in $\mathbb{R}^2$ given by
\begin{equation*}
	G_s(z)=\left\{
	\begin{array}{lll}
		\frac{1}{2\pi}\ln \frac{1}{|z|}, \ \ \ \ \ \ \ \ \ \ \ \ \ \ \ \ \ \ \ \ \ & \text{if} \ \ s=1;\\
		\frac{c_s}{|z|^{2-2s}}, \ \ \ c_s=\frac{\Gamma(1-s)}{2^{2s}\pi\Gamma(s)}, & \text{if} \ \ \frac{1}{2}\le s<1.
	\end{array}
	\right.
\end{equation*}
If we set
\begin{equation*}
	\mathbf{u}=\nabla^\perp\mathcal{G}_s\omega,
\end{equation*}
then by \eqref{1-1} and \eqref{1-2}, we can recover the volocity field
\begin{equation*}
	\mathbf{v}(x,t)=Q_{\alpha t}\mathbf{u}(Q_{-\alpha t}x)
\end{equation*}
and the first equation in \eqref{1-1} is reduced to a stationary equation
\begin{equation}\label{1-8}
\nabla^\perp(\mathcal{G}_s\omega+\frac{\alpha}{2} |x|^2)\cdot\nabla\omega=0.
\end{equation}
Hence it is natural to introduce the weak formulation of \eqref{1-8}, namely, $\omega$ satisfies
\begin{equation}\label{1-9}
	\int_{\mathbb{R}^2}\omega\nabla^\perp(\mathcal{G}_s\omega+\frac{\alpha}{2} |x|^2)\cdot\nabla \varphi  dx=0, \ \ \ \forall\,\varphi\in C_0^\infty(\mathbb{R}^2).
\end{equation}
The vortex solutions we will construct belong to $L^p(\mathbb{R}^2)$ for some $p>1$ and they are also of compact support. For the Euler equation, the integral in \eqref{1-9} makes sense if $p\ge 4/3$ by the standard theory of regularity for elliptic equations. It is slightly more subtle for the case $s\in [1/2,1)$. Due to the standard elliptic theory for Riesz potentials the corresponding functions $\mathcal{G}_s\omega\in \dot{W}^{2s,p}(\mathbb{R}^2)$, so the integral in \eqref{1-9} makes sense if $p\ge 2$. Notice that if $s<1/2$, the regularity is not sufficient to provide a rigorous meaning to \eqref{1-9}. For this reason, we restrict our attention to $s\in[1/2, 1]$. 

 Next, let us derive the equation satisfied by traveling-wave solutions. Up to a rotation, we may assume, without loss of generality, that these waves have a negative speed $-W$ in the vertical direction, so that
 \begin{equation*}
   \vartheta(x,t)=\omega(x_1, x_2+Wt).
 \end{equation*}
In this setting, the first equation in \eqref{1-1} is also reduced to a stationary equation
\begin{equation}\label{t1-1}
\nabla^\perp(\mathcal{G}_s\omega-Wx_1)\cdot\nabla\omega=0,
\end{equation}
which has a weak form
\begin{equation}\label{t1-2}
	\int_{\mathbb{R}^2}\omega\nabla^\perp(\mathcal{G}_s\omega-Wx_1)\cdot\nabla \varphi dx=0, \ \ \ \forall\,\varphi\in C_0^\infty(\mathbb{R}^2).
\end{equation}

From the above derivation, one sees that the construction of rotating and traveling-wave solutions for the gSQG equation can be reduced to solving stationary equations \eqref{1-8} and \eqref{t1-1}. In \cite{Ar1, Ar2, Ar3}, Arnol$'$d  proposed a variational principle which asserts that a stationary flow has extremal kinetic energy with respect to flows with equimeasurable vorticity. He also carried out a Lyapunov-type stability analysis. This consideration (in patch setting) can be traced back to the work of Kelvin \cite{Tho}.  Arnol$'$d's idea was adapted by Benjamin \cite{Ben} to study three-dimensional axisymmetric vortex rings in a uniform flow. Thanks to these works, the role of rearrangements in the theory of vorticity has attracted considerable attention and it has been proven successful in the study of the existence and stability for steady solutions; see, e.g., \cite{Bad1, Bad2, Bu, Bu0, Bu1, Bu2, Bu5, Bu6, Bu7, Bu8, Cao5, Dek1, Dek2, Dou, Elc1, Elc2, T1, Wan}. However, we would like to point out that these works mainly focused on the Euler equation and limited work has been done for the gSQG equation. The purpose of our paper is to construct rotating and traveling-wave solutions for the gSQG equation by following Arnol$'$d's idea.

To state our main results, we need to introduce some notations.
We denote by $B_R(x)$ the open ball in $\mathbb{R}^2$ of center $x$ and radius $R>0$. If $\Omega\subset\mathbb{R}^2$ is measurable then $\text{meas}\,({\Omega})$ denotes the two-dimensional Lebesgue measure of $\Omega$. $\textbf{1}_\Omega$ denotes the characteristic function of $\Omega\subset\mathbb{R}^2$. The right half-plane is denoted by $\mathbb{R}^2_+:=\{(x_1,x_2)\in \mathbb{R}^2: x_1>0\}$. Let $\xi$ be a non-negative Lebesgue integrable function on $\mathbb{R}^2$, we denote by $\mathcal{R}(\xi)$ the set of (equimeasurable) rearrangements of $\xi$ on $\mathbb{R}^2$ defined by
\begin{equation*}
  \mathcal{R}(\xi)=\Big\{\zeta\in L^1(\mathbb{R}^2):\zeta\ge 0~ \text{and}~ \text{meas}\left(\{x: \zeta(x)>\tau\}\right)=\text{meas}\left(\{x: \xi(x)>\tau\}\right), \forall\, \tau>0  \Big\}.
\end{equation*}
Let $\xi^*$ denote the (unique) radially symmetric nonincreasing rearrangement of $\xi$ centered at the origin. In the usual polar coordinates $(r,\theta)$, we shall say that $\xi$ is Steiner symmetric with respect to $\theta=0$ if $\xi$ for the variable $\theta$ is the unique even function such that
\begin{equation*}
  \xi(r,\theta)>\tau\ \ \ \text{if and only if}\ \ \ |\theta|<\frac{1}{2}\,\text{meas}\left\{\theta'\in (-\pi,\pi):\xi(r,\theta')>\tau \right\},
\end{equation*}
for any positive numbers $r$ and $\tau$, and any $-\pi<\theta<\pi$. For more information on rearrangements, see \cite{Lie}. We denote by $\text{supp}(\xi)$ the (essential) support of $\xi$. Set
\begin{equation*}
  \mathcal{U}_N:=\Big\{(r\cos\theta,r\sin\theta)\in \mathbb{R}^2:-\frac{\pi}{N}<\theta<\frac{\pi}{N} \Big\}.
\end{equation*}

We are now in a position to state our main results. Our first main result establishes the existence co-rotating vortices with $N$-fold symmetry for the gSQG equation in the context of rearrangements. More precisely, we have the following theorem:
\begin{theorem}\label{thm1}
Suppose $1/2\le s\le 1$. Let $N\ge2$ be a integer. Let non-negative $\xi \in L^p(\mathbb{R}^2)$ with $p\ge 4/3$ if $s=1$ and $p=\infty$ if $1/2\le s<1$. Suppose $\int_{\mathbb{R}^2}\xi dx=1$ and $\text{meas}\left(supp (\xi)\right)=\pi$.

Then there exists $\varepsilon_0>0$ such that for any $\varepsilon\in (0,\varepsilon_0]$, \eqref{1-1} has a global rotating solution $\vartheta_\varepsilon (x,t)$ with the following properties:
\begin{itemize}
  \item[(i)]$\vartheta_\varepsilon (x,t)=\omega_{ro, \varepsilon}(Q_{-\alpha_\varepsilon t}x)$, where angular velocity $\alpha_\varepsilon \in \mathbb{R}$ and $\omega_{ro, \varepsilon}\in L^p(\mathbb{R}^2)$ is a weak solution to \eqref{1-8} in the sense of \eqref{1-9}. Moreover, $\omega_{ro, \varepsilon}$ is $N$-fold symmetric, namely
      \begin{equation*}
      \omega_{ro, \varepsilon}(x)=\omega_{ro, \varepsilon}(Q_{\frac{2\pi}{N}}x)
      \end{equation*}
      for any $x\in \mathbb{R}^2$.
  \item[(ii)] Any single period of $\omega_{ro, \varepsilon}$ is a rearrangement of $\xi_\varepsilon(x)\equiv\varepsilon^{-2}\xi({x}/{\varepsilon})$. More precisely, let $\omega_{\varepsilon}=\omega_{ro, \varepsilon}\textbf{1}_{\mathcal{U}_N}$ be the 0-th part of $\omega_{ro, \varepsilon}$, then $\omega_{\varepsilon}\in \mathcal{R}(\xi_\varepsilon)$. Moreover, $\omega_{\varepsilon}$ is Steiner symmetric with respect to $\theta=0$.
  \item[(iii)]There exists some nonnegative and nondecreasing function $f_\varepsilon:\mathbb{R}\to \mathbb{R}$ with $f_\varepsilon(0)=0$ satisfying $f_\varepsilon(\tau)>0$ if $\tau>0$, such that
       \begin{equation*}
    \omega_{ro, \varepsilon}=f_{\varepsilon}\circ\psi_\varepsilon,
       \end{equation*}
   where
      \begin{equation*}
			\psi_\varepsilon(x)=\mathcal{G}_s\omega_{ro, \varepsilon}+\frac{\alpha_\varepsilon}{2}|x|^2-\mu_\varepsilon
	  \end{equation*}
		for some $\mu_\varepsilon\in \mathbb{R}$.
  \item[(iv)] One has
      \begin{equation*}
   \omega_{ro, \varepsilon}(x)\rightharpoonup \sum_{k=0}^{N-1}\delta_{Q_{\frac{2k\pi}{N}}(1,0)}\ \ \text{as}\ \ \varepsilon\to 0^+,
 \end{equation*}
 where the convergence is in the sense of measures. Moreover, there exists a constant $R_0>0$ independent of $\varepsilon$ such that
 \begin{equation*}
   \text{supp}(\omega_{ro, \varepsilon})\subset \bigcup^{N-1}_{k=0}B_{R_0\varepsilon}\left(Q_{\frac{2k\pi}{N}}(1,0)\right).
 \end{equation*}
 \item[(v)] As $\varepsilon \to 0^+$, it holds
  \begin{equation}\label{angu}
    \alpha_\varepsilon\to \frac{N-1}{4\pi}\ \ \text{if}\ \ s=1;\ \ \alpha_\varepsilon\to \sum\limits_{k=1}^{N-1}\frac{c_s(1-s)}{|(1,0)-Q_{\frac{2k\pi}{N}}(1,0)|^{2-2s}}\ \ \text{if}\ \ \frac{1}{2}\le s<1.
  \end{equation}
  In addition, concerning $\mu_\varepsilon$, we have
    \begin{equation*}
      \mu_\varepsilon=\frac{1}{2\pi}\ln\frac{1}{\varepsilon}+O(1)\ \ \text{if}\ \ s=1;\ \ \mu_\varepsilon=\frac{\mathcal{G}_s\xi^*\big|_{|x|=1}}{\varepsilon^{2-2s}}+o\left(\frac{1}{\varepsilon^{2-2s}}\right)\ \ \text{if}\ \ \frac{1}{2}\le s<1.
    \end{equation*}
  \item[(vi)] Let $\tilde{\omega}_{\varepsilon}(x)=\varepsilon^2\omega_{\varepsilon}(\varepsilon x)$ be the rescaled version of $\omega_{\varepsilon}$. Then there exists a translation $\mathcal{T}_\varepsilon$ on $\mathbb{R}^2$ such that as $\varepsilon\to 0^+$, $\mathcal{T}_\varepsilon\circ\omega_{\varepsilon} \to \xi^*$ in $L^q(\mathbb{R}^2)$ for every $q\in [1,p)$.
\end{itemize}
\end{theorem}

\begin{remark}
  The parameter $\varepsilon$ is introduced for technical reasons. Since the domain considered here is unbounded, one need to overcome the lack of compactness. A natural strategy is to resort to concentration. The solutions will concentrate near some points when $\varepsilon$ is sufficiently small. This trick has been widely used in the construction of solutions; see for example \cite{Ao, Cao4, Cao5, CWZ, Elc1, Elc2, Gar2, Go, HM, T1, T2}. Additionally, in the limit $\varepsilon \to 0^+$, we obtain a desingularization of $N$ point vortices located at the vertex of a regular polygon with $N$ sides (also called Thomson polygon). The point vortex model has received much attention in vorticity dynamics; see, e.g., \cite{Lin1, Lin2, March, Ros} and the references therein. The study of equal point vortices located at the vertices of a polygon, which rotates around its center, can be traced back to the work of Lord Kelvin 1878 and Thomson 1883. We refer the interested readers to \cite{Kur, New} for more information. We point out that the limiting angular velocity in \eqref{angu} does correspond exactly to the speed of rotation of $N$ point vortices located at the vertex of a regular polygon with $N$ sides evolving according to the gSQG equation, see \cite{Gar2}.
\end{remark}

\begin{remark}
  The support of $\omega_{ro, \varepsilon}$ has at least $N$ connected components when $\varepsilon$ is sufficiently small. Using scaling techniques, (vi) can be fed back into the stream function $\psi_\varepsilon$ by the standard elliptic theory. This fact can be used to further investigate the geometry the boundary of $\text{supp}(\omega_{ro, \varepsilon})$. For instance, for the Euler equation one can show that the boundary of $\text{supp}(\omega_{\varepsilon})$ converges to the unit circle in $C^1$ sense when $\varepsilon\to 0^+$; see, e.g., \cite{T1, T2}. The question that whether $\text{supp}(\omega_{\varepsilon})$ is convex is also interesting. However we are not able to answer this question here. For the regularity of $f_\varepsilon$, we refer to \cite{Elc3}. It is easy to prove that $f_\varepsilon$ and $\omega_{ro, \varepsilon}$ are both continuous provided $\xi$ is continuous.
\end{remark}

\begin{remark}
  If $\xi$ is a characteristic function of some measurable set, then the solutions in the above theorem are $N$ co-rotating vortex patches with $N$-fold symmetry. In this way we reobtained the known results in \cite{Gar2, Go, T2}. Our first main result can be viewed as a generalization of these results. In \cite{HM}, Hmidi and Mateu obtained the existence of co-rotating and counter-rotating vortex pairs by a perturbative argument. Compared to the perturbative method, our method seems to has stronger physical motivation. We construct solutions by maximizing kinetic energy among flows with prescribed distribution of vorticity and prescribed angular momentum (see Sections 2 and 3 below). As remarked by Benjamin \cite{Ben}, this variational principle is especially natural since it involves only quantities that enjoy Lagrangean conservation. Moreover, it is reasonable to speculate that extrema of a variational problem formulated entirely in terms of conserved physical quantities should be stable; see \cite{Abe, Ar1, Ar2, Ar3, Ben, Bu5, Bu6, Bu7, Wan}.
\end{remark}

Our second main result is on the existence of travelling-wave solutions. For the sake of simplicity, we focus on translating vortex pairs which are symmetric about the $x_2$-axis. More precisely, we have the following theorem:
\begin{theorem}\label{thm2}
Suppose $1/2\le s\le 1$. Let $W>0$ be given. Let non-negative $\xi \in L^p(\mathbb{R}^2)$ with $p\ge 4/3$ if $s=1$ and $p=\infty$ if $1/2\le s<1$. Suppose $\int_{\mathbb{R}^2}\xi dx=1$ and $\text{meas}\left(supp (\xi)\right)=\pi$.

Then there exists $\varepsilon_0>0$ such that for any $\varepsilon\in (0,\varepsilon_0]$, \eqref{1-1} has a global travelling-wave solution $\vartheta_\varepsilon (x,t)$ with the following properties:
\begin{itemize}
\item[(i)]$\vartheta_\varepsilon (x,t)=\omega_{tr,\varepsilon}(x_1,x_2+Wt)$, where $\omega_{tr,\varepsilon}\in L^p(\mathbb{R}^2)$ is a weak solution to \eqref{t1-1} in the sense of \eqref{t1-2}. Moreover, $\omega_{tr,\varepsilon}$ is an odd function with respect to the variable $x_1$.
\item[(ii)]Let $\omega_\varepsilon=\omega_{tr,\varepsilon}\textbf{1}_{\mathbb{R}^2_+}$ be the right half of $\omega_{tr,\varepsilon}$. Then $\omega_\varepsilon \in \mathcal{R}(\xi_\varepsilon)$ with $\xi_\varepsilon(x)\equiv\varepsilon^{-2}\xi({x}/{\varepsilon})$. Moreover, $\omega_{\varepsilon}$ is symmetric decreasing in $x_2$. That is, $\omega_\varepsilon(x_1,x_2)=\omega_\varepsilon(x_1, -x_2)\ge \omega_\varepsilon(x_1, {x}_2')$ for all $x_1\ge0$ and $0\le x_2\le {x}_2'$.
\item[(iii)] There exists some nonnegative and nondecreasing function $f_\varepsilon:\mathbb{R}\to \mathbb{R}$ with $f_\varepsilon(0)=0$ satisfying $f_\varepsilon(\tau)>0$ if $\tau>0$, such that
       \begin{equation*}
    \omega_{tr, \varepsilon}=f_{\varepsilon}\circ\psi_{tr,\varepsilon},
       \end{equation*}
    for all $(x_1,x_2)\in \mathbb{R}^2_+$, where
      \begin{equation*}
			\psi_\varepsilon(x)=\mathcal{G}_s\omega_{tr, \varepsilon}-Wx_1-\mu_\varepsilon
	  \end{equation*}
		for some $\mu_\varepsilon\in \mathbb{R}$.
\item[(iv)]Denote
		\begin{equation*}
			d=\left(\frac{1}{4\pi W}\frac{\Gamma(2-s)}{\Gamma(s)}\right)^{\frac{1}{3-2s}}, \ b_1=d\mathbf{e}_1,\ b_2=-d\mathbf{e}_1,\ \mathbf{e}_1=(1,0).
		\end{equation*}
	    One has
		\begin{equation*}
   \omega_{\text{tr}, \varepsilon}(x)\rightharpoonup \delta(x-b_1)-\delta(x-b_2)\ \ \ \text{as}\ \ \varepsilon\to 0^+,
 \end{equation*}
 where the convergence is in the sense of measures. Moreover, there exists a constant $R_0>0$ independent of $\varepsilon$ such that
 \begin{equation*}
   \text{diam}\left(\text{supp}(\omega_\varepsilon)\right)\le R_0\varepsilon.
 \end{equation*}
\item[(v)]
As $\varepsilon\to 0^+$, we have
    \begin{equation*}
      \mu_\varepsilon=\frac{1}{2\pi}\ln\frac{1}{\varepsilon}+O(1)\ \ \text{if}\ \ s=1;\ \ \mu_\varepsilon=\frac{\mathcal{K}_s\xi^*\big|_{|x|=1}}{\varepsilon^{2-2s}}+o\left(\frac{1}{\varepsilon^{2-2s}}\right)\ \ \text{if}\ \ \frac{1}{2}\le s<1.
    \end{equation*}
  \item[(vi)] Let $\tilde{\omega}_{\varepsilon}(x)=\varepsilon^2\omega_{\varepsilon}(\varepsilon x)$ be the rescaled version of $\omega_{\varepsilon}$. Then there exists a translation $\mathcal{T}_\varepsilon$ on $\mathbb{R}^2$ such that as $\varepsilon\to 0^+$, $\mathcal{T}_\varepsilon\circ\omega_{\varepsilon} \to \xi^*$ in $L^q(\mathbb{R}^2)$ for every $q\in [1,p)$.
	\end{itemize}
\end{theorem}

\begin{remark}
For the Euler equation (i.e. $s=1$), the above result is essentially contained in \cite{Bu0, Elc2, T1}. However, we include it here for the sake of completeness. To the best of our knowledge, there are not many known traveling-wave solutions to the gSQG equation; see \cite{Ao, Go0, Gra, HM}.
We provide here a new class of traveling-wave solutions to the gSQG equation.
\end{remark}

\begin{remark}
In the limit $\varepsilon\to 0^+$, we find that two point vortices with equal magnitude and opposite signs at distance $2d$ exhibit a uniform translating motion with the speed $$W=\frac{\Gamma(2-s)}{4\pi \Gamma(s) d^{3-2s}}.$$ This classical result is well-known in the literature, see \cite{ Ao, March} for example. In other words, we obtain a desingularization of a pair of point vortices with equal magnitude and opposite signs.
\end{remark}

\section{Construction of co-rotating vortices for the gSQG equation}\label{Sec2}
In this section, we provide a variational construction of co-rotating vortices with $N$-fold symmetry for the gSQG equation. Consider the kinetic energy of the fluid
\begin{equation*}
  {E}_s(\omega):=\frac{1}{2}\int_{\mathbb{R}^2}\int_{\mathbb{R}^2}G_s(x-x')\omega(x)\omega(x')dxdx',
\end{equation*}
and its angular momentum
\begin{equation*}
  L(\omega)=\int_{\mathbb{R}^2}|x|^2\omega(x)dx.
\end{equation*}
These quantities are conserved for sufficiently regular solutions to \eqref{1-1} (see for example \cite{Buc}). Following idea of Arnol$'$d, we will use the following energy maximization principle: \emph{Find a maximizer of the energy ${E}_s$ restricted on the set of all rearrangements of a given profile function $\xi$ with prescribed angular momentum}.

Thanks to the $N$-fold symmetry of the desired solution, the kinetic energy may be rewritten as
\begin{equation*}
  E_s(\omega)=\frac{N}{2}\int_{\mathcal{U}_N}\int_{\mathcal{U}_N}K_s(x,x')\omega(x)\omega(x')dxdx',
\end{equation*}
where the kernel $K_s$ is given by
\begin{equation*}
  K_s(x,x')=\sum_{k=0}^{N-1}G_s\left(x, Q_{\frac{2k\pi}{N}}x'\right).
\end{equation*}
For more information about kernel $K_s$, we refer to \cite{Go, T2} (see also the appendix below). Based on this observation, we shall restrict the construction to only one vortex inside the angular sector $\mathcal{U}_N$. For convenience, let us introduce the energy functional
\begin{equation*}
  \mathcal{E}_s(\omega)=\frac{1}{2}\int_{\mathcal{U}_N}\int_{\mathcal{U}_N}K_s(x,x')\omega(x)\omega(x')dxdx'
\end{equation*}
and the function
\begin{equation*}
  \mathcal{K}_s\omega(x)=\int_{\mathcal{U}_N}K_s(x,x')\omega(x')dx'.
\end{equation*}
 Because of some essential differences between Newtonian potential ($s=1$) and Riesz potential ($1/2\le s<1$), we consider two cases separately. In subsection \ref{s1}, we first construct co-rotating vortices with $N$-fold symmetry for the Euler equation. This situation is relatively simple and instructive. We will deal with the remaining cases $1/2\le s<1$ in subsection \ref{s2}.
\subsection{The gSQG equation with $s=1$}\label{s1}
In this subsection, we consider the Euler equation. Let nonnegative $\xi \in L^p(\mathbb{R}^2)$ with $p\ge 4/3$ satisfy $\int_{\mathbb{R}^2}\xi dx=1$ and $\text{meas}\left(supp (\xi)\right)=\pi$. Let $N\ge2$ be an integer. Let $\varepsilon>0$ and $\xi_\varepsilon(x)=\varepsilon^{-2}\xi({x}/{\varepsilon})$. It is easy to see that $\int_{\mathbb{R}^2}\xi_\varepsilon dx=1$ and $\text{meas}\left(supp (\xi_\varepsilon)\right)=\pi\varepsilon^2$. Let
\begin{equation*}
	S:=\left\{(r\cos\theta,r\sin\theta)\in \mathbb{R}^2:\frac{1}{2}< r<\frac{3}{2},\ \ -\frac{\pi}{2N}<\theta<\frac{\pi}{2N} \right\}
\end{equation*}
and
\begin{equation*}
  \mathcal{A}_{\varepsilon}:=\left\{\omega\in \mathcal{R}(\xi_\varepsilon):~ \omega=0\ \text{a.e. on}\ \mathbb{R}^2\backslash S \right\}, \ \ \  \mathcal{B}:=\left\{\omega\in L^1(S):L(\omega)=1\right\}.
\end{equation*}
Note that $\mathcal{A}_{\varepsilon}$ is not empty when $\varepsilon$ is sufficiently small. Inspired by Turkington's work \cite{T2}, we seek a maximizer of $\mathcal{E}_1$ relative to $\mathcal{A}_{\varepsilon}\cap \mathcal{B}$. However, the set $\mathcal{A}_{\varepsilon}\cap \mathcal{B}$ is not weakly compact in general, we thus first extend the class of admissible function for our maximization. Let $\mathcal{A}_{\varepsilon,4/3}$ denote the weak closure of $\mathcal{A}_{\varepsilon}$ in $L^{4/3}(S)$. We shall work in the set $\bar{\mathcal{A}_{\varepsilon}}:=\mathcal{A}_{\varepsilon,4/3}\cap \mathcal{B}$. We first have
\begin{lemma}\label{lem2-1}
	$\mathcal{E}_1$ attains its maximum value over $\bar{\mathcal{A}}_\varepsilon$ at some $\omega_\varepsilon$, which is Steiner symmetric with respect to $\theta=0$.
\end{lemma}
\begin{proof}
Notice that $K_1(\cdot,\cdot)\in L^q(S\times S)$ for any $1\le q <+\infty$. It is easy to check that $\mathcal{E}_1$ is bounded from above over $\bar{\mathcal{A}}_\varepsilon$. Let $\{\omega_j\}\subset \bar{\mathcal{A}}_\varepsilon$ be a sequence such that as $j\to+\infty$
\begin{equation*}
  \mathcal{E}_1(\omega_j)\to \sup_{\bar{\mathcal{A}}_\varepsilon}\mathcal{E}_1.
\end{equation*}
Up to a subsequence, we may assume that as $j\to +\infty$, $\omega_j\to \omega_\varepsilon \in \bar{\mathcal{A}}_\varepsilon$ weakly in $L^{4/3}(\mathbb{R}^2)$. Since $K_1(\cdot,\cdot)\in L^4(S\times S)$, we deduce that
\begin{equation*}
  \lim_{j\to +\infty}\mathcal{E}_1(\omega_j)=\mathcal{E}_1(\omega_\varepsilon).
\end{equation*}
This means that $\omega_\varepsilon$ is a maximizer. Moreover, by Lemma \ref{A2} in the appendix, we may further assume that $\omega_\varepsilon$ is Steiner symmetric with respect to $\theta=0$. Indeed, if otherwise, then we can replace $\omega_\varepsilon$ with its own angular Steiner symmetrization, which is still a maximizer. The proof is thus completed.
\end{proof}

Now, we further show that maximizer $\omega_\varepsilon$ in fact belongs to $\mathcal{A}_\varepsilon\cap \mathcal{B}$.
\begin{lemma}\label{lem2-2}
 There holds $\omega_\varepsilon\in\mathcal{A}_\varepsilon \cap \mathcal{B}$. Moreover, there exists a nonnegative and nondecreasing function $f_\varepsilon: \mathbb{R}\to \mathbb{R}$ with $f_\varepsilon(0)=0$ satisfying $f_\varepsilon(\tau)>0$ if $\tau>0$, such that
\begin{equation}\label{2-7}
		\omega_\varepsilon(x)=f_\varepsilon(\mathcal{K}_1\omega_\varepsilon(x)+\frac{\alpha_\varepsilon}{2}|x|^2-\mu_\varepsilon),\ \ \forall\,x\in S,
\end{equation}
for some Lagrange multipliers $\alpha_\varepsilon, \mu_\varepsilon \in \mathbb{R}$.
\end{lemma}
\begin{proof}
Thanks to the work of Burton \cite{Bu, Bu2}, we know that $\bar{\mathcal{A}}_\varepsilon$ is a convex set. So for each $\omega\in \bar{\mathcal{A}}_\varepsilon$, it holds $\omega_\tau:=\omega_\varepsilon+\tau(\omega-\omega_\varepsilon)\in \bar{\mathcal{A}}_\varepsilon$ for any $\tau\in [0,1]$. Since $\omega_\varepsilon$ is a maximizer, we have
	\begin{equation*}
		\frac{d}{d\tau}\bigg|_{\tau=0^+}\mathcal{E}_1(\omega_\tau)=\int_{S}(\omega-\omega_\varepsilon)\mathcal{K}_1\omega_\varepsilon dx\le 0,
	\end{equation*}
    which yields
    \begin{equation*}
    	\int_{S}\omega\mathcal{K}_1\omega_\varepsilon dx\le \int_{S}\omega_\varepsilon \mathcal{K}_1\omega_\varepsilon dx.
    \end{equation*}
    By Lemma \ref{A3} in the appendix, there exists some $\alpha_\varepsilon \in \mathbb{R}$ such that $\omega_\varepsilon$ is the maximizer of the linear functional
    \begin{equation*}
    	\mathcal{E}_{1,\ell}(\omega)=\int_S\omega(x)\left(\mathcal{K}_1\omega_\varepsilon(x)+\frac{\alpha_\varepsilon}{2}|x|^2\right)dx
    \end{equation*}
    relative to $\mathcal{A}_{\varepsilon,4/3}$. On the other hand, using the fact that $\omega_\varepsilon$ is Steiner symmetric with respect to $\theta=0$, it is easy to verify that $\mathcal{K}_1\omega_\varepsilon$ is strictly symmetric decreasing with respect to $\theta$ in $S$. It follows that every level set of $\mathcal{K}_1\omega_\varepsilon+{\alpha_\varepsilon}|x|^2/{2}$ in $S$ has measure zero. By Lemma \ref{A4} we have $\omega_\varepsilon\in \mathcal{A}_{\varepsilon}$, and there exists a nondecreasing function $\tilde{f}_\varepsilon: \mathbb{R}\to \mathbb{R}$, such that for any $x\in S$,
\begin{equation*}
		\omega_\varepsilon(x)=\tilde{f}_\varepsilon(\mathcal{K}_1\omega_\varepsilon(x)+\frac{\alpha_\varepsilon}{2}|x|^2).
\end{equation*}
Let
\begin{equation*}
  \mu_\varepsilon:=\sup\left\{\mathcal{K}_1\omega_\varepsilon(x)+\frac{\alpha_\varepsilon}{2}|x|^2:\ x\in S\ \text{s.t.}\ \omega_\varepsilon(x)=0 \right\}\in \mathbb{R},
\end{equation*}
and $f_\varepsilon(\cdot)=\max\{\tilde{f}_\varepsilon(\cdot+\mu_\varepsilon),0\}$. We then have $\omega_\varepsilon(x)=f_\varepsilon(\mathcal{K}_1\omega_\varepsilon(x)+{\alpha_\varepsilon}|x|^2/2-\mu_\varepsilon)$ for any $x\in S$. Moreover, by the definition of $\mu_\varepsilon$ and the continuity of $\mathcal{K}_1\omega_\varepsilon(x)$, we have $f_\varepsilon(0)=0$ and $f(\tau)>0$ if $\tau>0$. The proof is thus complete.
\end{proof}

For later discussion, let $\psi_\varepsilon=\mathcal{K}_1\omega_\varepsilon(x)+{\alpha_\varepsilon}|x|^2/2-\mu_\varepsilon$ be the corresponding stream function. We note that \eqref{2-7} is not yet sufficient to provide a dynamically possible steady vortex flow. To get a desired solution, we need to prove that the support of $\omega_\varepsilon$ is away from the boundary of $S$; see \cite{T2}. We will show that this is the case when $\varepsilon$ is sufficiently small. It is based on the observation that in order to maximize energy, the support of a maximizer can not be too scattered. We will reach this goal by several steps. We begin by giving a lower bound of $\mathcal{E}_1(\omega_\varepsilon)$. For convenience we will use $C$ to denote various positive constants not depending on $\varepsilon$ that may change from line to line.
\begin{lemma}\label{ener1}
  $\mathcal{E}_1(\omega_\varepsilon)\ge \frac{1}{4\pi}\ln\frac{1}{\varepsilon}-C$.
\end{lemma}
\begin{proof}
  The key idea is to choose a suitable test function. Let
\begin{equation}\label{2-14}
		\hat\omega_\varepsilon=\xi^\ast_\varepsilon(\cdot-\hat x_\varepsilon),
\end{equation}
  for some $\hat x_\varepsilon=(r_\varepsilon,0)$ such that $\hat\omega_\varepsilon \in \bar{\mathcal{A}}_\varepsilon$. It is easy to see that $r_\varepsilon\in \left(1-C\varepsilon,1+\varepsilon\right)$ when $\varepsilon$ is small. A direct calculation yields $\mathcal{E}_1(\hat\omega_\varepsilon)\ge \frac{1}{4\pi}\ln\frac{1}{\varepsilon}-C$. Recalling $\omega_\varepsilon$ is a maximizer, we have $\mathcal{E}_1(\omega_\varepsilon)\ge \mathcal{E}_1(\hat\omega_\varepsilon)$ and the proof is complete.
\end{proof}

The following result shows that the support of $\omega_\varepsilon$ is mostly concentrated.
\begin{lemma}\label{lem2-3}
	For every $0<\sigma<1$, there exist a set $\Omega_\sigma\subset S$ and a constant $R_\sigma>1$ independent of $\varepsilon$, such that $$\int_{\Omega_\sigma}\omega_\varepsilon dx\ge 1-\sigma,\ \ \text{diam} (\Omega_\sigma)\le R_\sigma\varepsilon,$$ and $\textbf{1}_{\Omega_\sigma}$ is Steiner symmetric with respect to $\theta=0$.
\end{lemma}
\begin{proof}
	By Lemma \ref{ener1}, we have
	\begin{equation*}
		\frac{1}{4\pi}\ln\frac{1}{\varepsilon}-C\le \mathcal{E}_1(\omega_\varepsilon)\le \frac{1}{4\pi}\int_S\int_S\ln\frac{1}{|x-x'|}\omega_\varepsilon(x)\omega_\varepsilon(x')dxdx'+C,
	\end{equation*}
    which implies
    \begin{equation}\label{2-15}
    	\int_S\int_S\ln\frac{\varepsilon}{|x-x'|}\omega_\varepsilon(x)\omega_\varepsilon(x')dxdx'\ge-C.
    \end{equation}
    On the other hand, by a simple rearrangement inequality (see \cite{Lie}, \S 3.4), we have for any $x\in S$ and $1<R<\infty$,
    \begin{equation}\label{a2-15}
    	\int_{|x-x'|<R\varepsilon}\ln\frac{\varepsilon}{|x-x'|}\omega_\varepsilon(x')dx'\le \int_{|y|<R}\ln\frac{1}{|y|}\xi^*(y)dy\le C||\xi||_{L^{4/3}}\le C.
    \end{equation}
    Then \eqref{2-15} and \eqref{a2-15} combine to give
    \begin{equation}\label{2-16}
      -C\le\int\int_{|x-x'|>R\varepsilon}\ln\frac{\varepsilon}{|x-x'|}\omega(x)\omega(x')dxdx'\le \ln\frac{1}{R}\int\int_{|x-x'|>R\varepsilon}\omega_\varepsilon(x)\omega_\varepsilon(x')dxdx'.
    \end{equation}
    Let
    \begin{equation*}
   	\Omega(R)=\left\{x\in S: \ \int_{|x-x'|>R\varepsilon}\omega_\varepsilon(x')dx'\le\frac{1}{2}\right\}.
    \end{equation*}
    It follows from \eqref{2-16} that
    \begin{equation}\label{2-18}
    	\int_{S\setminus\Omega(R)}\omega_\varepsilon(x)dx\le \frac{2C}{\ln R}.
    \end{equation}
    For every $0<\sigma<1$, we set $\Omega_\sigma=\Omega(R_\sigma/2)$ with $R_\sigma=2e^{2C/\sigma}$. From \eqref{2-18} we have
	\begin{equation*}
		\int_{\Omega_\sigma}\omega_\varepsilon(x) dx\ge 1-\sigma.
	\end{equation*}
    By virtue of the $\theta$-symmetrization of $\omega_\varepsilon$, we see that $\textbf{1}_{\Omega_\sigma}$ is also Steiner symmetric with respect to $\theta=0$. We now prove $\text{diam} (\Omega_\sigma)\le R_\sigma \varepsilon$. Suppose it were not true. Then we can find two points $z_1,z_2\in \Omega_\sigma$ such that
	\begin{equation*}
		B_{{R_\sigma\varepsilon}/{2}}(z_1)\cap B_{R_\sigma\varepsilon/2}(z_2)=\varnothing.
	\end{equation*}
    Recalling the definition of $\Omega_\sigma$, we derive
    \begin{equation*}
      1=\int_S\omega_\varepsilon(x)dx\ge \left(\int_{B_{R_\sigma\varepsilon/2}(z_1)}+\int_{B_{R_\sigma\varepsilon/2}(z_2)}\right)\omega_\varepsilon(x)dx>1,
    \end{equation*}
    which leads to a contradiction and the proof is thus complete.
\end{proof}

We now turn to estimate the angular velocity $\alpha_\varepsilon$. For convenience, let us introduce
\begin{equation*}
    \kappa(x,x')=\kappa(r,r',\theta-\theta')=-\frac{1}{2}\left(\frac{1}{r}\partial_rK_1(r,r',\theta-\theta')+\frac{1}{r'}\partial_{r'}K_1(r,r',\theta-\theta')\right).
\end{equation*}
From Turkington \cite{T2}, we know that
\begin{equation}\label{qq}
  0\le \kappa(x,x')\le \frac{N}{\pi}, \ \ \forall\,x, x' \in S.
\end{equation}
The following lemma shows that $\alpha_\varepsilon$ is uniformly bounded for $\varepsilon$.
\begin{lemma}\label{lem2-4}
	Provided that $\varepsilon$ is sufficiently small, it holds
	\begin{equation}\label{2-22}
		\alpha_\varepsilon=\int_S\int_S\kappa(x,x') \omega_\varepsilon(x)\omega_\varepsilon(x')dxdx'.
	\end{equation}
    Consequently, we have
    \begin{equation*}
    	0<\alpha_\varepsilon<\frac{N}{\pi}.
    \end{equation*}
\end{lemma}
\begin{proof}
	Set
	\begin{equation*}
		\rho:=\min\left\{\frac{1}{6}, \ \frac{1}{2}\sin\left(\frac{\pi}{2N}\right)\right\}.
	\end{equation*}
    Thanks to the constraint $L(\omega_\varepsilon)=1$, it is easy to see that if $\sigma$ is fixed small enough (independent of $\varepsilon$, for example, say $\sigma \le 4\rho/(15+4\rho)$) and $\varepsilon$ is sufficiently small then
\begin{equation*}
    	\Omega_\sigma\subset B_{R_\sigma\varepsilon}(x^*_\varepsilon)\subset \left\{x\in S: 1-\frac{\rho}{2}<|x|<1+\frac{\rho}{2}\right\}
    \end{equation*}
    for some $x^*_\varepsilon=(r^*_\varepsilon,0)$. It follows that
    \begin{equation}\label{2-25}
         \mathcal{K}_1\omega_\varepsilon(x)\ge\frac{1-\sigma}{2\pi}\ln\frac{1}{\varepsilon}-C \ \ \ \text{whenever}\ \ x\in B_{R_\sigma\varepsilon}(x^*_\varepsilon),
    \end{equation}
and
    \begin{equation}\label{2-26}
    	\mathcal{K}_1\omega_\varepsilon(x)\le\frac{\sigma}{2\pi}\ln\frac{1}{\varepsilon}+C, \ \ \ \text{whenever}\ \ x\in S\backslash B_\rho\left((1,0)\right).
    \end{equation}
    Let
    \begin{equation*}
    	S_1:=\left\{x\in S: \frac{1}{2}+\frac{1}{6}<|x|<1-\frac{1}{6}\right\} \ \ \text{and} \ \ S_2:=\left\{x\in S: 1+\frac{1}{6}<|x|<\frac{3}{2}-\frac{1}{6}\right\}.
    \end{equation*}

    We now prove that $\psi_\varepsilon<0$ in $S'=S_1\cup S_2$ if $\varepsilon$ is sufficiently small. We argue by contradiction. If the statement was false, then there exists a point $x\in S'$ satisfying $\psi_\varepsilon(x)\ge 0$. Then for any $y\in S$ such that $\psi_\varepsilon(y)\le 0$ we have clearly
    \begin{equation}\label{2-28}
      0\ge \psi_\varepsilon(y)-\psi_\varepsilon(x)=\mathcal{K}_1\omega_\varepsilon(y)-\mathcal{K}_1\omega_\varepsilon(x)+\frac{\alpha_\varepsilon}{2}\left(|y|^2-|x|^2\right).
    \end{equation}
    When $\varepsilon$ is small enough, we can find a point $y^1\in S\backslash B_\rho\left((1,0)\right)$ such that $\psi_\varepsilon(y_1)\le 0$, and which satisfies the two conditions
    \begin{equation*}
      \text{sign}(\alpha_\varepsilon)=\text{sign}(|y^1|^2-|x|^2)\ \ \text{and}\ \ \left||y^1|^2-|x|^2\right|\ge \frac{1}{12}.
    \end{equation*}
    Substituting $y^1$ for $y$ in the inequality \eqref{2-28}, we obtain $|\alpha_\varepsilon|\le 24 \left(\mathcal{K}_1\omega_\varepsilon(x)-\mathcal{K}_1\omega_\varepsilon(y^1)\right)$. Combining this with \eqref{2-26}, we get
    \begin{equation}\label{2-29}
      |\alpha_\varepsilon|\le C_1\sigma \ln\frac{1}{\varepsilon}+C,
    \end{equation}
    for some positive constant $C_1$ independent of $\sigma$ and $\varepsilon$. On the other hand, we may choose $y^2\in B_{R_\sigma\varepsilon}(x^*_\varepsilon)$ such that $\psi_\varepsilon(y^2)\le 0$ since $\text{meas}\left(\text{supp}(\omega_\varepsilon)\right)=\pi \varepsilon^2$. Now inequality \eqref{2-28} combined with \eqref{2-25} and \eqref{2-29} yields
        \begin{equation}\label{2-30}
    	\begin{split}
    	\frac{1-\sigma}{2\pi}\ln\frac{1}{\varepsilon}-C\le \mathcal{K}_1\omega_\varepsilon(y^2)&\le \mathcal{K}_1\omega_\varepsilon(x)+\frac{|\alpha_\varepsilon|}{2}\left(|y^2|^2-|x|^2\right)\\
    	&\le C_2\sigma\ln\frac{1}{\varepsilon}+C,
    	\end{split}
    \end{equation}
    for some positive constant $C_2$ independent of $\sigma$ and $\varepsilon$. Now let $\sigma$ be fixed so that ${1-\sigma}>{2\pi}C_2\sigma$. We get a contradiction from \eqref{2-30} when $\varepsilon$ is small enough. In other words, we have established that $\text{supp}(\omega_\varepsilon)\cap S'=\varnothing$.

    We now prove that $\alpha_\varepsilon$ is uniformly bounded with respect to $\varepsilon$. We decompose $\omega_\varepsilon$ as follows: $\omega_\varepsilon=\omega_{\varepsilon,1}+\omega_{\varepsilon,2}$ where
    \begin{equation*}
      \omega_{\varepsilon,1}=\omega_\varepsilon\textbf{1}_{\{1-\frac{1}{6}<|x|<1+\frac{1}{6}\}}, \ \ \ \omega_{\varepsilon,2}=\omega_\varepsilon\textbf{1}_{\{\frac{1}{2}<|x|<\frac{1}{2}+\frac{1}{6}\}{\cup}\{\frac{3}{2}-\frac{1}{6}<|x|<\frac{3}{2}\}}.
    \end{equation*}
    Then we have $\int_S \omega_{\varepsilon,1} dx\ge 1-\sigma$ and $\int_S \omega_{\varepsilon,2} dx\le\sigma$. By definition of $\psi_\varepsilon$ we have
    \begin{equation}\label{2-32}
   \frac{1}{r}(\psi_\varepsilon)_r=\frac{1}{r}(\mathcal{K}_1\omega_{\varepsilon,1})_r+\frac{1}{r}(\mathcal{K}_1\omega_{\varepsilon,2})_r+\alpha_\varepsilon.
    \end{equation}
    If we multiply \eqref{2-32} by $\omega_{\varepsilon,1}$ and integrate we get
     \begin{equation}\label{2-33}
     \int_S\frac{1}{r}(\psi_\varepsilon)_r\omega_{\varepsilon,1}dx=\int_S\frac{1}{r}(\mathcal{K}_1\omega_{\varepsilon,1})_r\omega_{\varepsilon,1}dx+\int_S\frac{1}{r}(\mathcal{K}_1\omega_{\varepsilon,2})_r\omega_{\varepsilon,1}dx+\alpha_\varepsilon\int_S\omega_{\varepsilon,1}dx.
    \end{equation}
    Let $F_\varepsilon(\tau):=\int_{0}^{\tau}f_\varepsilon(\tau')d\tau'$. Then
    \begin{equation}\label{2-34}
   \int_S\frac{1}{r}(\psi_\varepsilon)_r\omega_{\varepsilon,1}dx=\int_{-\frac{\pi}{2N}}^{\frac{\pi}{2N}}\int_{1-\frac{1}{6}}^{1+\frac{1}{6}}(F_\varepsilon(\psi_\varepsilon))_rdrd\theta=0.
    \end{equation}
    On the other hand, by definition of the kernel $\kappa$ we have
    \begin{equation}\label{2-35}
    	\int_S\frac{1}{r}(\mathcal{K}_1\omega_{\varepsilon,1})_r\omega_{\varepsilon,1}dx=-\int_S\int_S \kappa(x,x')\omega_{\varepsilon,1}(x)\omega_{\varepsilon,1}(x')dxdx'.
    \end{equation}
     Since $\text{dist}\left(\text{supp}(\omega_{\varepsilon,1}), \text{supp}(\omega_{\varepsilon,2})\right)\ge 1/6$, we have
    \begin{equation}\label{2-36}
    	\int_S\frac{1}{r}(\mathcal{K}_1\omega_{\varepsilon,2})_r\omega_{\varepsilon,1}dx\le C.
    \end{equation}
    It follows from \eqref{2-33}, \eqref{2-34}, \eqref{2-35} and \eqref{2-36} that
    \begin{equation}\label{2-37}
    	|\alpha_\varepsilon|\le C.
    \end{equation}

    With the above estimate of $\alpha_\varepsilon$ in hand, we can further show that $\text{supp}(\omega_\varepsilon)\subset B_\rho\left((1,0)\right)$ when $\varepsilon$ is sufficiently small. We argue by contradiction. If the assertion would not hold, then we can find a point $x\in S\backslash B_\rho\left((1,0)\right)$ such that $\psi_\varepsilon(x)>0$. On the other hand, we can find a point $y\in B_{R_\sigma\varepsilon}(x^*_\varepsilon)$ such that $\psi_\varepsilon(y)\le 0$. Hence we have
    \begin{equation*}
      0\ge \psi_\varepsilon(y)-\psi_\varepsilon(x)=\mathcal{K}_1\omega_\varepsilon(y)-\mathcal{K}_1\omega_\varepsilon(x)+\frac{\alpha_\varepsilon}{2}\left(|y|^2-|x|^2\right),
    \end{equation*}
    which, together with \eqref{2-25}, \eqref{2-26} and \eqref{2-37}, implies
    \begin{equation*}
      \frac{1-\sigma}{2\pi}\ln\frac{1}{\varepsilon}\le \frac{\sigma}{2\pi}\ln\frac{1}{\varepsilon}+C.
    \end{equation*}
    This is impossible if $0<\sigma<1/2$ and $\varepsilon$ is sufficiently small.

    Note that we have established that $\text{supp}(\omega_\varepsilon)\subset B_\rho\left((1,0)\right)$. So if we go back to \eqref{2-33} and \eqref{2-35}, we get
    \begin{equation*}
      \alpha_\varepsilon=-\int_S\frac{1}{r}(\mathcal{K}_1\omega_{\varepsilon,1})_r\omega_{\varepsilon,1}dx=\int_S\int_S \kappa(x,x')\omega_{\varepsilon}(x)\omega_{\varepsilon}(x')dxdx',
    \end{equation*}
    and the proof is complete.
    \end{proof}

In the proof of Lemma \ref{lem2-4}, we have established that $\text{dist}(\text{supp} (\omega_\varepsilon),\partial S)>0$ if $\varepsilon$ is sufficiently small. With this fact in hand, we can now show that $\omega_\varepsilon$ is a steady solution in the sense of \eqref{1-9}. More precisely, we have
\begin{lemma}\label{lem2-5}
Provided that $\varepsilon$ is sufficiently small, it holds
 \begin{equation*}
   \int_{\mathbb{R}^2}\omega_\varepsilon \nabla^\perp(\mathcal{K}_1\omega_\varepsilon+\frac{\alpha_\varepsilon}{2} |x|^2)\cdot\nabla \varphi  dx=0, \ \ \ \forall\,\varphi\in C_0^\infty(\mathbb{R}^2).
 \end{equation*}
\end{lemma}

\begin{proof}
  Recall that
  \begin{equation*}
		\omega_\varepsilon(x)=f_\varepsilon(\mathcal{K}_1\omega_\varepsilon(x)+\frac{\alpha_\varepsilon}{2}|x|^2-\mu_\varepsilon),\ \ \forall\,x\in S.
\end{equation*}
For any $\varphi\in C_0^\infty(\mathbb{R}^2)$, we can integrate by parts to obtain
\begin{equation*}
  \int_{\mathbb{R}^2}\omega_\varepsilon \nabla^\perp(\mathcal{K}_1\omega_\varepsilon+\frac{\alpha_\varepsilon}{2} |x|^2)\cdot\nabla \varphi dx=-\int_S F_\varepsilon(\psi_\varepsilon)(\partial_{x_2}\partial_{x_1}\varphi-\partial_{x_1}\partial_{x_2}\varphi)dx=0,
\end{equation*}
which completes the proof.
\end{proof}

Now we turn to study the asymptotic behaviors of $\omega_\varepsilon$ when $\varepsilon\to 0^+$. We begin by giving a lower bound of the Lagrange multiplier $\mu_\varepsilon$.
\begin{lemma}\label{ale1}
For $\varepsilon$ small, we have
\begin{equation*}
   \mu_\varepsilon\ge \frac{1}{2\pi}\ln\frac{1}{\varepsilon}-C.
\end{equation*}
\end{lemma}

\begin{proof}
  Notice that
  \begin{equation}\label{2-375}
  2\mathcal{E}_1(\omega_\varepsilon)=\int_S \omega_\varepsilon \mathcal{K}_1\omega_\varepsilon dx=\int_S\omega_\varepsilon\psi_\varepsilon dx+\mu_\varepsilon-\frac{\alpha_\varepsilon}{2}.
  \end{equation}
  In view of Lemmas \ref{ener1} and \ref{lem2-4}, it suffices to show that $\int_S\omega_\varepsilon\psi_\varepsilon dx\le C$. By definition, we have
  \begin{equation}\label{2-38}
   -\Delta \psi_\varepsilon=\omega_\varepsilon+O(1),\ \ \text{on}\ S.
  \end{equation}
   If we multiply both sides of \eqref{2-38} by $\psi_\varepsilon^+:=\max\{\psi_\varepsilon, 0\}$ and integrate by parts, we get
   \begin{equation}\label{ve1}
     \int_S |\nabla\psi^+_\varepsilon|^2 dx=\int_S \omega_\varepsilon \psi^+_\varepsilon dx+O(1)\int_S\psi_\varepsilon^+ dx.
   \end{equation}
   Let $q=p/(p-1)$. By H\"older inequality and the Sobolev imbedding theorem, recalling $\| \omega_\varepsilon\|_{L^p}=\|\xi_\varepsilon\|_{L^p}=\varepsilon^{-2/q}\|\xi\|_{L^p}$, we have
   \begin{equation}\label{ve2}
     \int_S \omega_\varepsilon \psi_\varepsilon^+ dx\le \| \omega_\varepsilon\|_{L^p}\|\psi_\varepsilon^+\|_{L^q}\le C\varepsilon^{-2/q}\|\xi\|_{L^p}[\text{meas}(\text{supp}(\psi_\varepsilon^+))]^{1/q}\|\nabla\psi_\varepsilon^+\|_{L^2}\le C\|\nabla\psi_\varepsilon^+\|_{L^2}
   \end{equation}
   and
   \begin{equation}\label{ve3}
     \int_S\psi_\varepsilon^+ dx\le C \|\nabla\psi_\varepsilon^+\|_{L^2}.
   \end{equation}
   If we combine \eqref{ve1}, \eqref{ve2} and \eqref{ve3}, we see that $\int_S\omega_\varepsilon\psi_\varepsilon dx\le C$. The proof is completed.
\end{proof}

Now we show that the size of $\text{supp}(\omega_\varepsilon)$ is of order $\varepsilon$. To this end, we first
recall an auxiliary lemma.
\begin{lemma}[\cite{CWZ}, Lemma 2.8]\label{lem2-6}
	Let $\Omega\subset \mathbb{R}^2$, $0<\varepsilon<1$, let non-negative $\omega_0\in L^1(\mathbb{R}^2)$, $\int_{\mathbb{R}^2} \omega_0(x)dx=1$ and $\|\omega_0\|_{L^p}\le C_3 \varepsilon^{-2(1-{1}/{p})}$ for some $1<p\le +\infty$ and $C_3>0$. Suppose for any $x\in \Omega$, it holds
	\begin{equation*}
		\ln\frac{1}{\varepsilon}\le \int_{\mathbb{R}^2} \ln\frac{1}{|x-x'|}\omega_0(x')dx'+C_4,
	\end{equation*}
	where $C_4$ is a positive constant.
	Then
	\begin{equation*}
		diam(\Omega)\le R\varepsilon,
	\end{equation*}
	for some constant $R>0$, which may depend on $C_3$, $C_4$, but not on $\varepsilon$.
\end{lemma}

\begin{lemma}\label{lem2-7}
	There exists some $R_0>1$ independent of $\varepsilon$, such that
	\begin{equation}\label{2-42}
		diam(\text{supp}(\omega_\varepsilon))\le R_0\varepsilon.
	\end{equation}
\end{lemma}
\begin{proof}
    First, we have
    \begin{equation*}
      \|\omega_\varepsilon\|_{L^p}=\|\xi_\varepsilon\|_{L^p}=\varepsilon^{-2(1-{1}/{p})}\|\xi\|_{L^p}.
    \end{equation*}
    Notice that for any $x\in \text{supp}(\omega_\varepsilon)$, there holds $\psi_\varepsilon(x)\ge 0$. It follows that
    \begin{equation*}
      \frac{1}{2\pi}\int_{S} \ln\frac{1}{|x-x'|}\omega_\varepsilon(x')dx'+C\ge \mathcal{K}_1\omega_\varepsilon(x)\ge \mu_\varepsilon-\frac{\alpha_\varepsilon}{2}|x|^2\ge \frac{1}{2\pi}\ln\frac{1}{\varepsilon}-C.
    \end{equation*}
    Using Lemma \ref{lem2-6}, we immediately get the desired result.
\end{proof}

As a consequence of Lemmas \ref{lem2-7}, we have
\begin{lemma}\label{lem2-8}
	As $\varepsilon\to 0^+$, one has
	\begin{equation*}
		\omega_\varepsilon\rightharpoonup \delta(x-(1,0)),
	\end{equation*}
     where the convergence is in the sense of measures.
\end{lemma}

Moreover, recalling $\eqref{2-22}$, we find upon calculation that (see also \cite{T2})
\begin{lemma}\label{lem2-9}
As $\varepsilon\to 0^+$, one has
\begin{equation*}
    	\alpha_\varepsilon\to\frac{N-1}{\pi}.
\end{equation*}
\end{lemma}

Using Riesz's rearrangement inequality, we can also sharpen Lemmas \ref{ener1} and \ref{ale1} as follows

\begin{lemma}\label{lem2-10}
One has
\begin{align}
\label{ad4}  \mathcal{E}_1(\omega_\varepsilon) &=\frac{1}{4\pi}\ln\frac{1}{\varepsilon}+O(1), \\
\label{ad5}  \mu_\varepsilon &=\frac{1}{2\pi}\ln\frac{1}{\varepsilon}+O(1).
\end{align}
\end{lemma}

\begin{proof}
  In view of \eqref{2-375}, it suffices to prove \eqref{ad4}.
  On the one hand, by Lemma \ref{ener1}, we already have
  \begin{equation}\label{2-44}
  \mathcal{E}_1(\omega_\varepsilon)\ge \frac{1}{4\pi}\ln\frac{1}{\varepsilon}-C.
  \end{equation}
  On the other hand, by Riesz's rearrangement inequality, we have
  \begin{equation}\label{2-45}
  \begin{split}
     \mathcal{E}_1(\omega_\varepsilon) & \le \frac{1}{2}\int_{\mathbb{R}^2}\int_{\mathbb{R}^2}K_1(x-x')\omega_\varepsilon(x)\omega_\varepsilon(x')dxdx' \\
       & \le \frac{1}{4\pi}\int_{\mathbb{R}^2}\int_{\mathbb{R}^2}\ln\frac{1}{|x-x'|}\omega_\varepsilon(x)\omega_\varepsilon(x')dxdx'+C  \\
       & \le \frac{1}{4\pi}\int_{\mathbb{R}^2}\int_{\mathbb{R}^2}\ln\frac{1}{|x-x'|}\omega^*_\varepsilon(x)\omega^*_\varepsilon(x')dxdx'+C\\
       & \le \frac{1}{4\pi}\ln\frac{1}{\varepsilon}+ \int_{\mathbb{R}^2}\int_{\mathbb{R}^2}\ln\frac{1}{|x-x'|}\xi^*(x)\xi^*(x')dxdx'+C\\
       & \le \frac{1}{4\pi}\ln\frac{1}{\varepsilon}+C.
  \end{split}
  \end{equation}
  Now \eqref{ad4} is just \eqref{2-44} and \eqref{2-45} combined.
\end{proof}

Finally, we turn to study the asymptotic shape of $\omega_\varepsilon$. To achieve this aim, we define the center of $\omega_\varepsilon$
\begin{equation*}
  x_\varepsilon:=\int_S x\, \omega_\varepsilon(x)dx
\end{equation*}
and its rescaled version
\begin{equation*}
	\zeta_\varepsilon(x):=\varepsilon^2\omega_\varepsilon(x_\varepsilon+\varepsilon x).
\end{equation*}
From the above results, we know that $\zeta_\varepsilon\in \mathcal{R}(\xi)$ and $\text{supp}(\zeta_\varepsilon)\subset B_{R_0}(0)$. The following result determines the asymptotic nature of ${\omega}_\varepsilon$ in terms of its rescaled version $\zeta_\varepsilon$.
\begin{lemma}\label{lem2-9}
	As $\varepsilon\to 0^+$, one has $\zeta_\varepsilon\to\xi^*$ in $L^q(\mathbb{R}^2)$ for any $q\in [1,p)$.
\end{lemma}
\begin{proof}
    Let $q\in(1,p)$, then $\|\zeta_\varepsilon\|_{L^q}=\|\xi\|_{L^q}$. We may assume that $\zeta_\varepsilon\to \zeta$ weakly in $L^q$ as $\varepsilon \to 0^+$. By the uniform convexity of $L^q$, it suffices to prove $\zeta=\xi^*$. By Riesz's rearrangement inequality we first have
	\begin{equation*}
		\int_{\mathbb{R}^2}\int_{\mathbb{R}^2}\ln\frac{1}{|x-y|}\zeta_\varepsilon(x)\zeta_\varepsilon(x')dxdx'\le \int_{\mathbb{R}^2}\int_{\mathbb{R}^2}\ln\frac{1}{|x-y|}\xi^\ast(x)\xi^\ast(x')dxdx',
	\end{equation*}
   and hence
   \begin{equation}\label{2-53}
     \int_{\mathbb{R}^2}\int_{\mathbb{R}^2}\ln\frac{1}{|x-y|}\zeta(x)\zeta(x')dxdx'\le \int_{\mathbb{R}^2}\int_{\mathbb{R}^2}\ln\frac{1}{|x-y|}\xi^\ast(x)\xi^\ast(x')dxdx'.
   \end{equation}
    On the other hand, let $\hat{\omega}_\varepsilon$ be defined by \eqref{2-14}, we have by $\mathcal{E}_\varepsilon(\omega_\varepsilon)\ge \mathcal{E}_\varepsilon(\hat{\omega}_\varepsilon)$
    \begin{equation*}
    	\int_{\mathbb{R}^2}\int_{\mathbb{R}^2}\ln\frac{1}{|x-y|}\zeta_\varepsilon(x)\zeta_\varepsilon(x')dxdx'\ge \int_{\mathbb{R}^2}\int_{\mathbb{R}^2}\ln\frac{1}{|x-y|}\xi^\ast(x)\xi^\ast(x')dxdx'+o(1).
    \end{equation*}
    Letting $\varepsilon\to 0^+$, we get
    \begin{equation}\label{2-55}
    	\int_{\mathbb{R}^2}\int_{\mathbb{R}^2}\ln\frac{1}{|x-y|}\zeta(x)\zeta(x')dxdx'= \int_{\mathbb{R}^2}\int_{\mathbb{R}^2}\ln\frac{1}{|x-y|}\xi^\ast(x)\xi^\ast(x')dxdx'
    \end{equation}
    By Lemma 3.2 in Burchard–Guo \cite{BG}, we know that there exists a translation $\mathcal{T}$ on $\mathbb{R}^2$ such that $\mathcal{T}\circ\zeta=\xi^*$. Since $$\int_{\mathbb{R}^2}x \zeta dx=\int_{\mathbb{R}^2}x \xi^* dx=(0,0),$$ we conclude that $\zeta=\xi^*$. Hence we have established that $\zeta_\varepsilon \to\xi^*$ in $L^q(\mathbb{R}^2)$ for any $q\in (1,p)$. Recall that $\text{supp}(\zeta_\varepsilon)\subset B_{R_0}(0)$. Using H\"older inequality, we also have $\zeta_\varepsilon \to\xi^*$ in $L^1(\mathbb{R}^2)$. The proof is thus complete.
\end{proof}
\begin{remark}
  From the above proof, one also see that if $p<+\infty$, then $\tilde{\omega}_\varepsilon \to\xi^*$ in $L^p(\mathbb{R}^2)$.
\end{remark}

\subsection{The gSQG equation with $1/2\le s<1$}\label{s2}
In this subsection, we consider the gSQG equation with $1/2\le s<1$. Let nonnegative $\xi \in L^\infty(\mathbb{R}^2)$ satisfy $\int_{\mathbb{R}^2}\xi dx=1$ and $\text{meas}\left(supp (\xi)\right)=\pi$. Let $N\ge2$ be an integer. Let $\varepsilon>0$ and $\xi_\varepsilon(x)=\varepsilon^{-2}\xi({x}/{\varepsilon})$. Let
\begin{equation*}
	S:=\left\{(r\cos\theta,r\sin\theta)\in \mathbb{R}^2:\frac{1}{2}< r<\frac{3}{2},\ \ -\frac{\pi}{2N}<\theta<\frac{\pi}{2N} \right\}.
\end{equation*}
Let
\begin{equation*}
  \mathcal{A}_{\varepsilon}:=\left\{\omega\in \mathcal{R}(\xi_\varepsilon):~ \omega=0\ \text{a.e. on}\ \mathbb{R}^2\backslash S \right\}, \ \ \  \mathcal{B}:=\left\{\omega\in L^1(S):L(\omega)=1\right\}.
\end{equation*}
The weak closure of $\mathcal{A}_{\varepsilon}$ in $L^4(S)$ is denoted by ${\mathcal{A}_{\varepsilon,4}}$. Set $\bar{\mathcal{A}_{\varepsilon}}={\mathcal{A}_{\varepsilon,4}} \cap \mathcal{B}$. Observe that $K_s \in L^{q}(S\times S)$ for every $1\le q<1/(1-s)$. Using the same arguments as employed in the preceding subsection, we can easily get the following result:
\begin{lemma}\label{lem3-1}
  $\mathcal{E}_s$ attains its maximum value over $\bar{\mathcal{A}}_\varepsilon$ at some $\omega_\varepsilon\in\mathcal{A}_\varepsilon$, which is Steiner symmetric with respect to $\theta=0$. Moreover, there exists a nonnegative and nondecreasing function $f_\varepsilon: \mathbb{R}\to \mathbb{R}$ with $f_\varepsilon(0)=0$ satisfying $f(\tau)>0$ if $\tau>0$, such that
\begin{equation*}
		\omega_\varepsilon(x)=f_\varepsilon(\mathcal{K}_s\omega_\varepsilon(x)+\frac{\alpha_\varepsilon}{2}|x|^2-\mu_\varepsilon),\ \ \forall\,x\in S,
\end{equation*}
for some Lagrange multipliers $\alpha_\varepsilon, \mu_\varepsilon \in \mathbb{R}$.
\end{lemma}
Let $\psi_\varepsilon=\mathcal{K}_s\omega_\varepsilon(x)+{\alpha_\varepsilon}|x|^2/2-\mu_\varepsilon$ be the corresponding stream function. The following lemma gives an asymptotic estimate of $\mathcal{E}_s(\omega_\varepsilon)$.
\begin{lemma}\label{lem3-3}
	One has
	\begin{equation}\label{3-10}
		\mathcal{E}_s(\omega_\varepsilon)= \frac{A_s}{\varepsilon^{2-2s}}+O(1),
	\end{equation}
	where
	\begin{equation*}
		A_s=\frac{c_s}{2}\int_{B_1(0)}\int_{B_1(0)}\frac{1}{|x-x'|^{2-2s}}\xi^\ast(x)\xi^\ast(x')dxdx'.
	\end{equation*}
\end{lemma}
\begin{proof}
Let $\hat{\omega}_\varepsilon$ be defined by \eqref{2-14}. A simple calculation yields
\begin{equation}\label{3-12}
  \mathcal{E}_s(\omega_\varepsilon)\ge \mathcal{E}_s(\hat{\omega}_\varepsilon)=\frac{A_s}{\varepsilon^{2-2s}}+O(1).
\end{equation}
 On the other hand, by Riesz's rearrangement inequality we have
 \begin{equation}\label{3-14}
    \mathcal{E}_s(\omega_\varepsilon)\le \frac{c_s}{2}\int_S\int_S\frac{1}{|x-x'|^{2-2s}}\xi^\ast_\varepsilon(x)\xi^\ast_\varepsilon(x')dxdx'+C=\frac{A_s}{\varepsilon^{2-2s}}+C.
    \end{equation}
   If we combine \eqref{3-12} and \eqref{3-14}, we obtain\eqref{3-10} and complete the proof.
\end{proof}

Unlike the procedure in subsection \ref{s1}, we first determine the asymptotic behaviour of ${\omega}_\varepsilon$ in terms of its rescaled version. We introduce the center of $\omega_\varepsilon$
\begin{equation*}
  x_\varepsilon:=\int_S x\, \omega_\varepsilon(x)dx
\end{equation*}
and its (first) rescaled version
\begin{equation*}
	\tilde{\zeta}_\varepsilon(x):=\varepsilon^2\omega_\varepsilon(x_\varepsilon+\varepsilon x).
\end{equation*}
Note that $\tilde{\zeta}_\varepsilon\in \mathcal{R}(\xi)$. We have the following result.

\begin{lemma}\label{lem3-4}
There exists $\{z^\varepsilon\}\subset \mathbb{R}^2$ such that as $\varepsilon\to 0^+$, $\tilde{\zeta}_\varepsilon(\cdot+z_\varepsilon)\to \xi^*$ in $L^q (\mathbb{R}^2)$ for every $1\le q<+\infty$.
\end{lemma}
\begin{proof}
  We first observe that
  \begin{equation*}
    \mathcal{E}_s(\omega_\varepsilon)=\frac{c_s}{2\varepsilon^{2-2s}}\int_{\mathbb{R}^2}\int_{\mathbb{R}^2}\frac{1}{|x-x'|^{2-2s}}\tilde{\zeta}_\varepsilon(x)\tilde{\zeta}_\varepsilon(x')dxdx'+O(1).
  \end{equation*}
  If we combine this with \eqref{3-10}, we see that
  \begin{equation}\label{3-16}
    \lim_{\varepsilon\to 0^+} \int_{\mathbb{R}^2}\int_{\mathbb{R}^2}\frac{1}{|x-x'|^{2-2s}}\tilde{\zeta}_\varepsilon(x)\tilde{\zeta}_\varepsilon(x')dxdx'=\int_{\mathbb{R}^2}\int_{\mathbb{R}^2}\frac{1}{|x-x'|^{2-2s}}\xi^\ast(x)\xi^\ast(x')dxdx'.
  \end{equation}
  Recall that $\tilde{\zeta}_\varepsilon\in \mathcal{R}(\xi)$. By Theorem 1 in Burchard--Guo \cite{BG}, we find that there exists $\{z_\varepsilon\}\subset \mathbb{R}^2$ such that
	\begin{equation}\label{3-17}
		\lim_{\varepsilon\to0^+} \int_{\mathbb{R}^2}\int_{\mathbb{R}^2}\left(\tilde \zeta_\varepsilon(x+z_\varepsilon)-\xi^\ast(x)\right)\frac{1}{|x-x'|^{2-2s}}\left(\tilde \zeta_\varepsilon(x'+z_\varepsilon)-\xi^\ast(x')\right)dxdx'=0.
	\end{equation}
Since $\xi \in L^1\cap L^\infty$, up to a subsequence we may assume that $\tilde \zeta_\varepsilon(\cdot+z_\varepsilon)\to \zeta$ weakly in $L^q(\mathbb{R}^2)$ as $\varepsilon\to 0^+$ for every $1<q<+\infty$. Then we have
\begin{equation}\label{3-18}
  \int_{\mathbb{R}^2}\int_{\mathbb{R}^2}\frac{1}{|x-x'|^{2-2s}}\zeta(x)\zeta(x')dxdx'\le \int_{\mathbb{R}^2}\int_{\mathbb{R}^2}\frac{1}{|x-x'|^{2-2s}}\xi^\ast(x)\xi^\ast(x')dxdx',
\end{equation}
and
\begin{equation}\label{3-19}
  \lim_{\varepsilon\to 0^+}\int_{\mathbb{R}^2}\int_{\mathbb{R}^2}\frac{1}{|x-x'|^{2-2s}}\tilde \zeta_\varepsilon(x+z_\varepsilon)\xi^*(x')dxdx'=\int_{\mathbb{R}^2}\int_{\mathbb{R}^2}\frac{1}{|x-x'|^{2-2s}}\zeta(x)\xi^\ast(x')dxdx'.
\end{equation}
If we combine \eqref{3-16}, \eqref{3-17}, \eqref{3-18} and \eqref{3-19}, we obtain
\begin{equation*}
  \int_{\mathbb{R}^2}\int_{\mathbb{R}^2}\left(\zeta(x)-\xi^\ast(x)\right)\frac{1}{|x-x'|^{2-2s}}\left(\zeta(x)-\xi^\ast(x')\right)dxdx'=0,
\end{equation*}
from which it follows that $\zeta=\xi^*$ (see \cite{Lie}, \S 9.8). Since $\|\zeta_\varepsilon\|_{L^q}=\|\xi^*\|_{L^q}$, we conclude that $\tilde \zeta_\varepsilon(\cdot+z_\varepsilon)\to \xi^*$ in $L^q(\mathbb{R}^2)$ as $\varepsilon\to 0^+$. It remains to show $\tilde \zeta_\varepsilon(\cdot+z_\varepsilon)\to \xi^*$ in $L^1(\mathbb{R}^2)$. Indeed, this can be obtained by Br$\acute{\text{e}}$zis-Lieb lemma \cite{Bre}, since $\|\zeta_\varepsilon\|_{L^1}=\|\xi^*\|_{L^1}=1$. The proof is thus complete.
\end{proof}

For further use, let $x^\ast_\varepsilon:=x_\varepsilon+\varepsilon z_\varepsilon$ be the modified center, and $\zeta_\varepsilon:=\tilde{\zeta}_\varepsilon(\cdot+z_\varepsilon)$. Note that
\begin{equation*}
  \int_{B_\varepsilon(x^\ast_\varepsilon)}\omega_\varepsilon(x)dx=\int_{B_1(0)}\zeta_\varepsilon(x)dx.
\end{equation*}
As a consequence of Lemma \ref{lem3-4}, we have
\begin{lemma}\label{lem3-5}
It holds
\begin{equation*}
    \int_{B_\varepsilon(x^\ast_\varepsilon)}\omega_\varepsilon(x)dx\to 1, \ \ \text{as}\ \varepsilon \to 0^+.
\end{equation*}
\end{lemma}
Moreover, by the constraint $L(\omega_\varepsilon)=1$ and the $\theta$-symmetrization of $\omega_\varepsilon$, we also have
\begin{lemma}\label{lem3-6}
  As $\varepsilon \to 0^+$, $x^\ast_\varepsilon\to (1,0)$.
\end{lemma}

Let
\begin{equation*}
  \Phi_\varepsilon(x)=c_s\int_{\mathbb{R}^2}\frac{\zeta_\varepsilon(x')}{|x-x'|^{2-2s}}dx',\ \ \ \ \Phi(x)=c_s\int_{\mathbb{R}^2}\frac{\xi^*(x')}{|x-x'|^{2-2s}}dx'.
\end{equation*}
Using Lemma \ref{lem3-4}, one can easily get the following result.
\begin{lemma}\label{lem3-7}
  As $\varepsilon \to 0^+$, $\Phi_\varepsilon\to \Phi$ in $L^\infty(\mathbb{R}^2)$.
\end{lemma}

Let $B_s=\Phi(x)\big{|}_{|x|=1}$ and
\begin{equation*}
		\rho:=\min\left\{\frac{1}{6}, \ \frac{1}{2}\sin\left(\frac{\pi}{2N}\right)\right\}.
\end{equation*}
Notice that
\begin{equation*}
  \mathcal{K}_s\omega_\varepsilon(x)=\frac{\Phi_\varepsilon\left({\varepsilon}^{-1}({x-x_\varepsilon^*})\right)}{\varepsilon^{2-2s}}+O(1),\ \ \ \text{on}\ S.
\end{equation*}
By Lemma \ref{lem3-7}, we know that for every $\sigma\in (0,1)$, there holds
   \begin{equation}\label{3-20}
         \mathcal{K}_s\omega_\varepsilon(x)\ge\frac{(1-\sigma)B_s}{\varepsilon^{2-2s}}-C \ \ \ \text{whenever}\ \ x\in B_{\varepsilon}(x^*_\varepsilon),
    \end{equation}
and
    \begin{equation}\label{3-21}
    	\mathcal{K}_s\omega_\varepsilon(x)\le\frac{\sigma B_s}{\varepsilon^{2-2s}}+C, \ \ \ \text{whenever}\ \ x\in S\backslash B_\rho\left((1,0)\right),
    \end{equation}
provided $\varepsilon$ is sufficiently small. With \eqref{3-20} and \eqref{3-21} in hand, one can use the same argument as in the proof of Lemma \ref{lem2-4} to obtain
\begin{equation}\label{3-22}
  \text{supp}(\omega_\varepsilon)\subset B_\rho\left((1,0)\right)\cup\left\{x: \frac{1}{2}<|x|<\frac{1}{2}+\frac{1}{6}\right\}\cup\left\{x: \frac{3}{2}-\frac{1}{6}<|x|<\frac{3}{2}\right\}.
\end{equation}
Now we study the Lagrange multiplier $\alpha_\varepsilon$. We have
\begin{lemma}\label{lem3-8}
	As $\varepsilon\to 0^+$, it holds
	\begin{equation}\label{3-24}
		\alpha_\varepsilon\to \sum\limits_{k=1}^{N-1}\frac{c_s(1-s)}{|(1,0)-Q_{\frac{2k\pi}{N}}|^{2-2s}}.
	\end{equation}
\end{lemma}
\begin{proof}
  Let $y\in \mathbb{R}^2$ be arbitrary. Thanks to \eqref{3-22}, we may choose a function $\varphi\in C_0^\infty\left(B_{2\rho}\left((1,0)\right)\right)$ so that $\varphi(x)=y\cdot x$ on $B_\rho\left((1,0)\right)$. Let $F_\varepsilon(\tau):=\int_{0}^{\tau}f_\varepsilon(\tau')d\tau'$. Then $F_\varepsilon(\psi_\varepsilon)=0$ on $\partial B_{2\rho}\left((1,0)\right)$. By integration by parts we get
\begin{equation}\label{3-25}
  \int_{\mathbb{R}^2}\omega_\varepsilon \nabla^\perp(\mathcal{K}_s\omega_\varepsilon+\frac{\alpha_\varepsilon}{2} |x|^2)\cdot\nabla \varphi dx=-\int_S F_\varepsilon(\psi_\varepsilon)(\partial_{x_2}\partial_{x_1}\varphi-\partial_{x_1}\partial_{x_2}\varphi)dx=0.
\end{equation}
On the other hand, we have
\begin{equation}\label{3-26}
 \int_{\mathbb{R}^2}\omega_\varepsilon\nabla^\perp\left(\frac{\alpha_\varepsilon}{2} |x|^2\right)\cdot\nabla \varphi dx=\alpha_\varepsilon\int_{B_\rho\left((1,0)\right)}\omega_\varepsilon(x) x^\perp\cdot y dx,
\end{equation}
and
\begin{equation}\label{3-27}
\begin{split}
   \int_{\mathbb{R}^2}&\omega_\varepsilon\nabla^\perp\mathcal{K}_s\omega_\varepsilon \cdot\nabla \varphi dx \\
    =&\int_{B_\rho\left((1,0)\right)}\omega_\varepsilon(x)\nabla^\perp_x\left(\int_{B_\rho\left((1,0)\right)}+\int_{S\backslash B_\rho\left((1,0)\right)}G_s(x,x')\omega_\varepsilon(x')dx'\right)\cdot y dx\\
      =& \int_{B_\rho\left((1,0)\right)}\omega_\varepsilon(x)\nabla^\perp_x \left(\int_{B_\rho\left((1,0)\right)}G_s(x,x')\omega_\varepsilon(x')dx'\right)\cdot y dx+o(1)\\
      = &-c_s(1-s)\int_{B_\rho\left((1,0)\right)}\int_{B_\rho\left((1,0)\right)}\omega_\varepsilon(x)\omega_\varepsilon(x')\sum_{k=1}^{N-1}\frac{(x-Q_{\frac{2k\pi}{N}}x')^\perp}{|x-Q_{\frac{2k\pi}{N}}x'|^{4-2s}}\cdot(y-Q_{\frac{2k\pi}{N}}y')dx'dx+o(1).
\end{split}
\end{equation}
In view of Lemmas \ref{lem3-5} and \ref{lem3-6}, it then follows from \eqref{3-25}, \eqref{3-26} and \eqref{3-27} that
\begin{equation*}
  \alpha_\varepsilon y_2=-c_s(1-s)\sum_{k=1}^{N-1}\frac{\left((1,0)-Q_{\frac{2k\pi}{N}}(1,0)\right)^\perp}{|(1,0)-Q_{\frac{2k\pi}{N}}(1,0)|^{4-2s}}\cdot(y-Q_{\frac{2k\pi}{N}}y')+o(1).
\end{equation*}
 Taking $y=(0,-1)=(1,0)^\perp$, the conclusion \eqref{3-24} follows.
\end{proof}

From Lemma \ref{lem3-8}, we see that $\alpha_\varepsilon$ is uniformly bounded. Now we turn to the Lagrange multiplier $\mu_\varepsilon$. We have the following asymptotic estimate.

\begin{lemma}\label{lem3-9}
  There holds
  \begin{equation}\label{3-28}
    \mu_\varepsilon=\frac{B_s}{\varepsilon^{2-2s}}+o\left(\frac{1}{\varepsilon^{2-2s}}\right).
  \end{equation}
\end{lemma}

\begin{proof}
  The proof relies on a small scale asymptotic analysis. Let $\Psi_\varepsilon(x)=\psi_\varepsilon(x_\varepsilon^*+\varepsilon x)$ be the rescaled version of $\psi_\varepsilon$. First, note that
  \begin{equation*}
    \Psi_\varepsilon(x)=\frac{\Phi_\varepsilon(x)}{\varepsilon^{2-2s}}-\mu_\varepsilon+O(1),\ \ \text{on}\ B_2(0).
  \end{equation*}
  By virtue of Lemma \ref{lem3-5}, for $\varepsilon$ small, we can find a point $\tilde{x}_\varepsilon \in \partial\left(\text{supp}(\omega_\varepsilon)\right)\cap B_2(0)$ so that $\Psi_\varepsilon(\tilde{x}_\varepsilon)=0$. Moreover, up to a subsequence, we may assume $\tilde{x}_\varepsilon \to \tilde{x}$ as $\varepsilon \to 0^+$. By Lemma \ref{lem3-7}, we have
  \begin{equation*}
    \begin{split}
       0=\varepsilon^{2-2s}\Psi_\varepsilon(\tilde{x}_\varepsilon)&=\Phi(\tilde{x}_\varepsilon)-\varepsilon^{2-2s}\mu_\varepsilon+o(1) \\
         &= \Phi(\tilde{x})-\varepsilon^{2-2s}\mu_\varepsilon+o(1),
    \end{split}
  \end{equation*}
  from which it follows that
  \begin{equation*}
    \mu_\varepsilon=\frac{\Phi(\tilde{x})}{\varepsilon^{2-2s}}+o\left(\frac{1}{\varepsilon^{2-2s}}\right).
  \end{equation*}
  We now claim that $|\tilde{x}|=1$. Indeed, suppose not, then we have either $|\tilde{x}|>1$ or $|\tilde{x}|<1$.  Set $\rho_1:=(1+|\tilde{x}|)/2$. If $|\tilde{x}|>1$, then there exists a $\delta>0$ small (independent of $\varepsilon$) such that
  \begin{equation}\label{3-30}
    \Phi(x)>\Phi(\tilde{x})+\delta,\ \ \forall\,x\in B_{\rho_1}(0).
  \end{equation}
   This implies
  \begin{equation*}
    \varepsilon^{2-2s}\Psi_\varepsilon(x)=\Phi(x)-\varepsilon^{2-2s}\mu_\varepsilon+o(1)\ge \delta+o(1)>0,\ \ \ \forall\,x\in B_{\rho_1}(0),
  \end{equation*}
  whence $B_{\rho_1}(0)\subset \text{supp}(\omega_\varepsilon)$. However, this contradicts the fact that $\text{supp}(\omega_\varepsilon)=\pi$. If $|\tilde{x}|<1$, we can use a similar argument to get $\text{supp}(\omega_\varepsilon)\cap B_2(0)\subset B_{\rho_1}(0)$. This is contrary to Lemma \ref{lem3-5}. Hence we have $|\tilde{x}|=1$ and therefore $\Phi(\tilde{x})=B_s$. The proof is thus complete.
  \end{proof}

Now we are able to show that the diameter of $\text{supp}(\omega_\varepsilon)$ is of order $\varepsilon$.
\begin{lemma}\label{le9}
  $diam(supp(\omega_\varepsilon))\le R_0\varepsilon$ for some $R_0>0$ not depending on $\varepsilon$.
\end{lemma}

\begin{proof}
Note that, for each $x_0\in \text{supp}(\omega_\varepsilon)$, we have
  \begin{equation*}
    \frac{\Phi_\varepsilon\left({\varepsilon}^{-1}({x_0-x_\varepsilon^*})\right)}{\varepsilon^{2-2s}}=\mathcal{K}_s\omega_\varepsilon(x_0)+O(1)\ge \mu_\varepsilon +O(1)= \frac{B_s}{\varepsilon^{2-2s}}+o\left(\frac{1}{\varepsilon^{2-2s}}\right).
  \end{equation*}
  This implies $\Psi_\varepsilon\left({\varepsilon}^{-1}({x_0-x_\varepsilon^*})\right)= B_s+o(1)$. Combining this with Lemma \ref{lem3-7}, we get $\Phi\left({\varepsilon}^{-1}({x_0-x_\varepsilon^*})\right)= B_s+o(1)$. Since $\Phi$ is strictly symmetric decreasing, we must have $|{\varepsilon}^{-1}({x_0-x_\varepsilon^*})|\le 2$ when $\varepsilon$ is small enough. The proof is therefore completed by taking $R_0=2$.
  \end{proof}

Recalling Lemma \ref{lem3-6}, we have established that $\text{supp}(\omega_\varepsilon)\subset B_\rho\left((1,0)\right)$, and hence $\text{dist}\left(\text{supp}(\omega_\varepsilon), \partial S\right)>0$ when $\varepsilon$ is sufficiently small. With this fact in hand, arguing as in the proof of Lemma \ref{lem2-5}, we immediately get the following result.
\begin{lemma}
Provided that $\varepsilon$ is sufficiently small, it holds
 \begin{equation*}
   \int_{\mathbb{R}^2}\omega_\varepsilon \nabla^\perp(\mathcal{K}_s\omega_\varepsilon+\frac{\alpha_\varepsilon}{2} |x|^2)\cdot\nabla \varphi  dx=0, \ \ \ \forall\,\varphi\in C_0^\infty(\mathbb{R}^2).
 \end{equation*}
\end{lemma}

Now we are ready to prove Theorem \ref{thm1}.

\noindent{\bf Proof of Theorem \ref{thm1}:}
Let $\omega_{\text{ro},\varepsilon}(x)=\sum_{k=0}^{N-1}\omega_\varepsilon(Q_{\frac{2k\pi}{N}}x)$. Then the statements of Theorem \ref{thm1} follow from the above lemmas.
\qed

\section{Construction of translating vortex pairs for the gSQG equation}\label{Sec3}

In this section, we provide a variational construction of translating vortex pairs for the gSQG equation. Recall that the kinetic energy of the fluid is given by
\begin{equation*}
  {E}_s(\omega):=\frac{1}{2}\int_{\mathbb{R}^2}\int_{\mathbb{R}^2}G_s(x-x')\omega(x)\omega(x')dxdx'.
\end{equation*}
We also introduce the impulse of the fluid as follows:
\begin{equation*}
  \mathcal{I}(\omega)=\int_{\mathbb{R}^2}x_1 \omega(x)dx.
\end{equation*}
Inspired by Benjamin's work \cite{Ben}, we will use the following energy maximization principle in this situation: \emph{Find a maximizer of the energy ${E}_s-W\mathcal{I}$ restricted on the set of all rearrangements of a given profile function $\xi$ with given translational velocity $W>0$}.

For the sake of simplicity, we focus on translating vortex pairs which are symmetric about the $x_2$-axis. More precisely, we first assume
\begin{equation*}
  \omega(x_1,x_2)=-\omega(-x_1,x_2).
\end{equation*}
With this symmetry assumption in hand, the energy functional may be rewritten as
\begin{equation*}
  E_s(\omega)-W\mathcal{I}(\omega)=2\left(\frac{1}{2}\int_{\mathbb{R}^2_+}\int_{\mathbb{R}^2_+}P_s(x,x')\omega(x)\omega(x')dxdx'-W\int_{\mathbb{R}^2_+}x_1 \omega(x)dx \right),
\end{equation*}
where the kernel $P_s$ is equal to
\begin{equation*}
  P_s(x,x')=G_s(x,x')-G_s(\bar{x}, x'),\ \ \text{with}\ \ \bar{x}:=(-x_1, x_2).
\end{equation*}
Based on this fact, we shall restrict the construction to only one vortex on $\mathbb{R}^2_+$. For convenience, we set
\begin{equation*}
  \mathcal{F}_s(\omega)=\frac{1}{2}\int_{\mathbb{R}^2_+}\int_{\mathbb{R}^2_+}P_s(x,x')\omega(x)\omega(x')dxdx'-W\int_{\mathbb{R}^2_+}x_1 \omega(x)dx.
\end{equation*}
Let non-negative $\xi \in L^p(\mathbb{R}^2)$ with $p\ge 4/3$ if $s=1$ and $p=\infty$ if $1/2\le s<1$. Assume that $\int_{\mathbb{R}^2}\xi dx=1$ and $\text{meas}\left(supp (\xi)\right)=\pi$. Let $\varepsilon>0$ and $\xi_\varepsilon(x)=\varepsilon^{-2}\xi({x}/{\varepsilon})$. Let
\begin{equation*}
			d=\left(\frac{1}{4\pi W}\frac{\Gamma(2-s)}{\Gamma(s)}\right)^{\frac{1}{3-2s}}, \ b_1=d\mathbf{e}_1,\ b_2=-d\mathbf{e}_1,\ \mathbf{e}_1=(1,0).
		\end{equation*}
 We shall consider the class of admissible function as follows:
\begin{equation*}
  \mathcal{B}_{\varepsilon}:=\left\{\omega\in \mathcal{R}(\xi_\varepsilon):~ \omega=0\ \text{a.e. on}\ \mathbb{R}^2_+\backslash B_{d/2}(b_1) \right\}.
\end{equation*}

Since the remainder of the argument is analogous to that in Theorem \ref{thm1}, we will only give an outline of the proof.

The first step is to prove the existence of a maximizer for $\mathcal{F}_\varepsilon$ relative to $\mathcal{B}_{\varepsilon}$. To carry out a weak compactness argument, we need to extend the class of admissible functions for our maximization, since $\mathcal{B}_{\varepsilon}$ is not weakly compact in general. We first show that $\mathcal{F}_\varepsilon$ always attains its supremum relative to the closed convex hull (in a suitable $L^q$-space) of $\mathcal{B}_{\varepsilon}$. To complete the first step, we have to show that the energy maximizer $\omega_\varepsilon$ in fact belongs to $\mathcal{B}_{\varepsilon}$. This can be achieved by using a argument similar to the proof of Lemma \ref{lem2-2}. In this argument, angular Steiner symmetrization should be replaced by Steiner symmetrization with respect to $x_2=0$; see, e.g., \cite{Bu0, Nor}.

The second step is to prove that the support of $\omega_\varepsilon$ is away from the boundary of $B_{d/2}(b_1)$. Once this has been done, the proof of the theorem is complete. One can obtain this result by the same method as employed in the preceding section. The maximization of energy will force the support of $\omega_\varepsilon$ to be very concentrated when $\varepsilon$ tends to zero. The location of concentration was determined by the constraint $L(\omega_\varepsilon)=1$ in Section \ref{Sec2}. In this case, the situation is slightly different. The leading term of $\mathcal{F}_\varepsilon(\omega_\varepsilon)$ results in concentration. The second order term of $\mathcal{F}_\varepsilon(\omega_\varepsilon)$ determines the precise location of the asymptotic singular vortex pair. More precisely, the second order term of $\mathcal{F}_\varepsilon(\omega_\varepsilon)$ is related to the associated Kirchoff-Routh function $\mathcal{W}(\tau)\equiv G_s(2\tau)+2W\tau$ defined on $\mathbb{R}_+$; see \cite{Cao4, SV}. To maximize the energy, the location of concentration should be the minimum point of $\mathcal{W}$. It is easy to see that $\mathcal{W}$ has a unique global minimum point, $d$. In fact, this is also the reason why we choose $\mathcal{B}_{\varepsilon}$ as admissible functions.

\appendix

\section{Auxiliary results}

In this appendix, we collect some auxiliary results, which we have been used in the preceding sections. Recall that
\begin{equation*}
  \mathcal{U}_N=\Big\{(r\cos\theta,r\sin\theta)\in \mathbb{R}^2:-\frac{\pi}{N}<\theta<\frac{\pi}{N} \Big\}.
\end{equation*}
Consider a non-negative measurable function $\omega$ defined on $\mathcal{U}_N$. We define its angular Steiner symmetrization $\omega^\sharp$ to be the unique rearrangement of $\omega$ which is Steiner symmetric with respect to $\theta=0$. In other words, $\omega^\sharp$ is the unique even function for the variable $\theta$ such that
\begin{equation*}
  \omega^\sharp(r,\theta)>\tau\ \ \ \text{if and only if}\ \ \ |\theta|<\frac{1}{2}\,\text{meas}\left\{\theta'\in (-\frac{\pi}{N},\frac{\pi}{N}):\omega(r,\theta')>\tau \right\},
\end{equation*}
for any positive numbers $r$ and $\tau$, and any $-{\pi}/{N}<\theta<{\pi}/{N}$. Using the layer-cake representation of nonnegative measurable functions (see \cite{Lie}), we have
\begin{lemma}\label{A1}
  Let non-negative $\omega\in L^1_{\text{loc}}(\mathcal{U}_N)$ and let $g:\mathbb{R}_+\to \mathbb{R}_+$ be continuous. Then we have
  \begin{equation*}
    \int_{\mathcal{U}_N}g(r)\omega(r,\theta)rdrd\theta=\int_{\mathcal{U}_N}g(r)\omega^\sharp(r,\theta)rdrd\theta,
  \end{equation*}
  provided that these quantities are finite.
\end{lemma}

Recall that
\begin{equation*}
  \mathcal{E}_s(\omega)=\frac{1}{2}\int_{\mathcal{U}_N}\int_{\mathcal{U}_N}K_s(x,x')\omega(x)\omega(x')dxdx',
\end{equation*}
where the kernel $K_s$ is equal to
\begin{equation*}
  K_s(x,x')=\sum_{k=0}^{N-1}G_s\left(x, Q_{\frac{2k\pi}{N}}x'\right).
\end{equation*}
In the polar coordinates $(r,\theta)$, the kernel $K_s$ can be expressed as
\begin{equation*}
  K_s(r,\theta,r',\theta')=J_s(r,r',\theta-\theta').
\end{equation*}
Given two fixed numbers $1/2<r\neq r'<3/2$, the map $\tau \mapsto J_s(r,r',\tau)$ is smooth on $(-\pi/N, \pi/N)$. Moreover, we have
\begin{itemize}
  \item [(i)]$J_s(r,r',-\tau)=J_s(r,r',\tau)$\ \ \ \text{whenever}\ \ $|\tau|<\pi/N$,
  \item [(ii)]$\partial_\tau J_s(r,r',\tau)<0$\ \ \ \ \ \ \ \ \ \ \ \ \ \ \text{whenever}\ \ $0<\tau<\pi/N$.
\end{itemize}
For proofs, we refer to \cite{Go, T2}. Using these facts, one can easily get the following result which asserts that the energy $\mathcal{E}_s$ is increased by the angular Steiner symmetrization.

\begin{lemma}\label{A2}
  Let non-negative $\omega \in L^p(S)$ with $p\ge 4/3$ if $s=1$ and $p=\infty$ if $1/2\le s<1$. Then
  \begin{equation*}
    \mathcal{E}_s(\omega)\le \mathcal{E}_s(\omega^\sharp).
  \end{equation*}
\end{lemma}

Next, we state two useful results due to Burton \cite{Bu2}.
\begin{lemma}\label{A3}
	Let $\mathcal{C}$ be a convex set in a real vector space $\mathscr{X}$. Let $L_1$ and $L_2$ be linear functionals on $\mathscr{X}$, let $I\in\mathbb{R}$ and suppose there exist $e_1$ and $e_2$ in $\mathcal{C}$ such that $L_1(e_1)<I<L_1(e_2)$. Suppose $e_0\in \mathcal{C}$ is such that $L_2(u)\le L_2(e_0)$ for all $u \in\mathcal{C}$ satisfying $L_1(e_0)=I$. Then there is a real number $\lambda$ such that $e_0$ maximizes $L_2+\lambda L_1$ relative to $\mathcal{C}$.
\end{lemma}

Let $(\Omega,\nu)$ be a finite positive measure space. Let $\xi_0: \Omega \to \mathbb{R}$ be $\nu$-measurable. We denote by $\mathcal{R}_{\nu}(\xi_0)$ the set of (equimeasurable) rearrangements of $\xi_0$ with respect to the measure $\nu$. That is,
\begin{equation*}
\mathcal{R}_\nu(\xi_0)=\Big\{\xi:\Omega\to \mathbb{R}\ \  \text{measurable},\ \text{s.t.}\ \ \nu\left(\{x: \xi(x)>\tau\}\right)=\nu\left(\{x: \xi_0(x)>\tau\}\right), \forall\, \tau\in \mathbb{R}  \Big\}.
\end{equation*}
If $\xi_0 \in L^p(\Omega,\nu)$ for some $1\le p<+\infty$, we denote by $\overline{\mathcal{R}}_{\nu,p}(\xi_0)$ the closure of $\mathcal{R}_\nu(\xi_0)$ in the weak topology of $L^p(\Omega,\nu)$. Let $p'$ denote the conjugate exponent of $p$, that is, $1/p'+1/p=1$.
\begin{lemma}\label{A4}
	Let $(\Omega,\nu)$ be a finite positive measure space. Let $\xi_0: \Omega\to \mathbb{R}$ and $\zeta_0: \Omega\to \mathbb{R}$ be $\nu$-measurable functions, and suppose that every level set of $\zeta_0$ has zero measure. Then there is a nondecreasing function $f$ such that $f\circ \zeta_0$ is a rearrangement of $\xi_0$. Moreover, if $\xi_0 \in L^p(\Omega,\nu)$ for some $1\le p<+\infty$ and $\zeta_0\in L^{p'}(\Omega,\nu)$, then $f\circ \zeta_0$ is the unique maximizer of linear functional
	\begin{equation*}
		M(\xi):=\int_\Omega \xi(x)\zeta_0(x)\nu(dx)
	\end{equation*}
   relative to $\overline{\mathcal{R}}_{\nu,p}(\xi_0)$.
\end{lemma}

\phantom{s}
 \thispagestyle{empty}


\begin{thebibliography}{99}
    \bibitem{Abe}
    K. Abe and K. Choi, Stability of Lamb dipoles, Preprint	arXiv:1911.01795.
	
    \bibitem{Amb}
    A. Ambrosetti and J. Yang, Asymptotic behaviour in planar vortex theory, \textit{Atti Accad. Naz. Lincei Cl. Sci. Fis. Mat. Natur. Rend. Lincei (9) Mat. Appl.},  1(4)(1990), 285--291.


	\bibitem{Ao}
	W. Ao, J. D\'avila, L.D. Pino, M. Musso and J. Wei, Travelling and rotating solutions to the generalized inviscid surface quasi-geostrophic equation, arXiv:2008,12911.
	
    \bibitem{Ar1}
     V. I. Arnol'd, Conditions for nonlinear stability of stationary plane curvilinear flows of an ideal fluid, \textit{Soviet Math. Doklady} 162(1965), 773--777; Translation of \textit{Dokl. Akad. Nauk SSSR}, 162(1965), 975--998.

    \bibitem{Ar2}
     V. I. Arnol'd, Variational principles for three-dimensional steady-state flows of an ideal fluid, \textit{J. Appl. Math. Mech.},29(1965), 1002--1008; Translation of \textit{Prikl. Mat. Mekh.}, 29(1965), 846--851.

    \bibitem{Ar3}
     V. I. Arnol'd, On an a priori estimate in the theory of hydrodynamic stability, \textit{Amer. Math. Soc. Transl.} 79(1969), 267--269; Translation of \textit{Izv. Vyssh. Uchebn. Zaved. Mat.}, 5(1966), 3--5.

    \bibitem{Bad1}
    T. V. Badiani, Existence of steady symmetric vortex pairs on a planar domain with an obstacle, \textit{Math. Proc. Cambridge Philos. Soc.}, 123(1998), 365--384.

    \bibitem{Bad2}
     T. V. Badiani and G. R. Burton, Vortex rings in $\mathbb{R}^3$ and rearrangements, \textit{R. Soc. Lond. Proc. Ser. A Math. Phys. Eng. Sci.}, 457(2001), no. 2009, 1115--1135.

	
    \bibitem{Bar}
    C. Bardos, Existence et unicit$\acute{\text{e}}$ de la solution de l'$\acute{\text{e}}$quation d'Euler en dimension deux, \textit{J. Math. Anal. Appl.}, 40 (1972), 769--790.

    \bibitem{Ben}
     T. B. Benjamin, The alliance of practical and analytic insights into the nonlinear problems of fluid mechanics, \textit{Applications of Methods of Functional Analysis to Problems of Mechanics, Lecture Notes in Math.}, vol. 503, Springer--Verlag, Berlin, 1976, 8--29.



    \bibitem{Ber}
    A. L. Bertozzi and P. Constantin, Global regularity for vortex patches, \textit{Comm. Math. Phys.}, 152(1)(1993), 19--28.

    \bibitem{Bre}
    H. Br$\acute{\text{e}}$zis and E. A. Lieb, A relation between pointwise convergence of functions and convergence of functionals, \textit{Proc. Amer. Math. Soc.}, 88(1983), no. 3, 486--490.




    \bibitem{Buc}
    T. Buckmaster, S. Shkoller and V. Vicol, Nonuniqueness of weak solutions to the SQG equation, \textit{Comm. Pure Appl. Math.}, 72(2019), no. 9, 1809--1874.

    \bibitem{Burb}
    J. Burbea, Motions of vortex patches, \textit{Lett. Math. Phys.}, 6(1982), 1--16.


	\bibitem{BG}
	A. Burchard and Y. Guo, Compactness via symmetrization,
	\textit{J. Funct. Anal.}, 214(2004), 40-73.

    \bibitem{Bu}
    G. R. Burton, Rearrangements of functions, maximization of convex functionals, and vortex rings, \textit{Math. Ann.}, 276(1987), no. 2, 225--253.

    \bibitem{Bu0}
     G. R. Burton, Steady symmetric vortex pairs and rearrangements, \textit{Proc. R. Soc. Edinb., Sect. A}, 108(1988), 269-290.

    \bibitem{Bu1}
	G. R. Burton, Variational problems on classes of rearrangements and multiple configurations for steady vortices,
	\textit{Ann. Inst. H. Poincar\'e Anal. Non Lineair\'e}, 6(1989), no. 4, 295--319.

    \bibitem{Bu2}
    G. R. Burton, Rearrangements of functions, saddle points and uncountable families of steady configurations for a vortex, \textit{Acta Math.}, 163(1989), no. 3-4, 291--309.


    \bibitem{Bu3}
    G. R. Burton, Uniqueness for the circular vortex-pair in a uniform flow, \textit{Proc. Roy. Soc. London Ser. A}, 452(1996), no. 1953, 2343-2350.

    \bibitem{Bu4}
    G. R. Burton, Isoperimetric properties of Lamb’s circular vortex-pair, \textit{J. Math. Fluid
Mech.}, 7(2005), S68–S80.
	
	\bibitem{Bu5}
	G. R. Burton, Global nonlinear stability for steady ideal fluid flow in bounded planar domains,
	\textit{Arch. Rational Mech. Anal.}, 176(2005), no. 1, 149--163.

    \bibitem{Bu6}
    G. R. Burton, Compactness and stability for planar vortex-pairs with prescribed impulse, \textit{J. Differential Equations}, 270(2021), 547--572.

    \bibitem{Bu7}
    G. R .Burton, H. J. Nussenzveig Lopes and M. C. Lopes Filho, Nonlinear stability for steady vortex pairs, \textit{Comm. Math. Phys.}, 324(2013), 445--463.

    \bibitem{Bu8}
    G. R. Burton and J. F. Toland, Surface waves on steady perfect-fluid flows with vorticity, \textit{Comm. Pure Appl. Math.}, 64(2011), no. 7, 975--1007.


    \bibitem{Cao4}
    D. Cao, S. Lai and W. Zhan, Traveling vortex pairs for 2D incompressible Euler equations, Preprint arXiv:2012.10918.
	
    \bibitem{Cao5}
    D. Cao, J. Wan and W. Zhan, Desingularization of vortex rings in 3 dimensional Euler flows, \textit{J. Differential Equations}, 270(2021), 1258--1297.

	\bibitem{CWZ}
	D. Cao, G. Wang and W. Zhan, Desingularization of vortices for 2D steady Euler flows via the vorticity method,
	\textit{SIAM J. Math. Anal.}, 52(2020), no. 6, 5363--5388.
	
	\bibitem{Cas1}
	A. Castro, D. C\'ordoba and J. G\'omez-Serraon, Existence and regularity of rotating global solutions for the generalized surface quasi-geostrophic equations,
	\textit{Duke Math. J.}, 165(2016), no. 5, 935--984.
	
    \bibitem{Cas4}
    A. Castro, D. C\'ordoba and J. G\'omez-Serraon, Uniformly rotating analytic global patch solutions for active scalars, \textit{Ann. PDE}, 2(2016), no. 1, Art. 1, 34.

    \bibitem{Cas3}
    A. Castro, D. C\'ordoba, and J. G\'omez-Serrano, Uniformly rotating smooth solutions for the incompressible 2D Euler equations, \textit{Arch. Ration. Mech. Anal.}, 231(2019), no. 2, 719--785.


	\bibitem{Cas2}
	A. Castro, D. C\'ordoba, and J. G\'omez-Serrano, Global smooth solutions for the inviscid SQG equation,
	\textit{Mem. Amer. Math. Soc.}, 266(2020), no. 1292.




    \bibitem{Chae0}
    D. Chae, The quasi-geostrophic equation in the Triebel-Lizorkin spaces, \textit{Nonlinearity}, 16(2003), no. 2, 479--495.
	
	\bibitem{Chae}
	D. Chae, P. Constantin, D. Cordoba, F. Gancedo and J. Wu, Generalized surface quasi-geostrophic equations with singular velocities,
	\textit{Comm. Pure Appl. Math.}, 65(2012), no. 8, 1037--1066.
	

    \bibitem{Che}
    J.-Y. Chemin, Fluides Parfaits Incompressibles, Astérisque 230, 1995 (Perfect Incompressible Fluids translated by I. Gallagher and D. Iftimie, Oxford Lecture Series in Mathematics and Its Applications, Vol. 14, Clarendon Press-Oxford University Press, New York, 1998).



	\bibitem{Con}
	P. Constantin, A.J. Majda and E. Tabak, Formation of strong fronts in the 2-D quasigeostrophic thermal active scalar,
	\textit{Nonlinearity}, 7(1994), no. 6, 1495--1533.


    \bibitem{Cor1}
    D. C\'ordoba,  Nonexistence of simple hyperbolic blow-up for the quasi-geostrophic equation., \textit{Ann. of Math.}, (2)148(1998), no. 3, 1135--1152.

    \bibitem{Cor2}
    D. C\'ordoba, C. Fefferman, Growth of solutions for QG and 2D Euler equations, \textit{J. Amer. Math. Soc.}, 15(2002), no. 3, 665--670.


    \bibitem{Cor}
    D. C\'ordoba, M. A. Fontelos, A. M. Mancho and J. L. Rodrigo, Evidence of singularities for a family of contour dynamics equations, \textit{Proc. Natl. Acad. Sci. USA} 102(2005), 5949--5952.
	

    \bibitem{Deem}
    G. S. Deem and N. J. Zabusky, Vortex waves: stationary “V-states”, interactions, recurrence, and breaking, \textit{Phys. Rev. Lett.}, 40(13)(1978), 859--862.

    \bibitem{Dek1}
    J. Dekeyser, Asymptotic of steady vortex pair in the lake equation, \textit{SIAM J. Math. Anal.}, 51(2019), 1209--1237.

    \bibitem{Dek2}
    J. Dekeyser, Desingularization of a steady vortex pair in the lake equation, \textit{Potential Anal.}, 2020, https://doi.org/10.1007/s11118-020-09878-w.



    \bibitem{de1}
    F. de la Hoz, Z. Hassainia and T. Hmidi, Doubly connected V-states for the generalized surface quasigeostrophic equations, \textit{Arch. Ration. Mech. Anal.}, 220(3)(2016), 1209--1281.

    \bibitem{de2}
    F. de la Hoz, T. Hmidi, J. Mateu and J. Verdera, Doubly connected V-states for the planar Euler equations, \textit{SIAM J. Math. Anal.}, 48(3)(2016), 1892--1928.


    \bibitem{Del}
    J.-M., Delort, Existence de nappes de tourbillon en dimension deux, \textit{J. Amer. Math. Soc.}, 4(1991), no. 3, 553--586.


    \bibitem{Dip}
    R. J. DiPerna and A. J. Majda, Concentrations in regularizations for 2-D incompressible flow, \textit{Comm. Pure Appl. Math.}, 40(1987), no. 3, 301--345.

    \bibitem{Dou}
     R. J. Douglas, Rearrangements of functions on unbounded domains, \textit{Proc. R. Soc. Edinb., Sect. A}, 124(1994), 621--644.


    \bibitem{Dri}
    D. G. Dritschel, T. Hmidi and C. Renault, Imperfect bifurcation for the quasi-geostrophic shallow-water equations, \textit{Arch. Ration. Mech. Anal.}, 231(2019), 1853--1915.

       \bibitem{Elg}
    T. M. Elgindi and I.-J. Jeong, Symmetries and critical phenomena in fluids, \textit{Commun. Pure Appl. Math.}, 73(2)(2020), 257--316.


    \bibitem{Elc1}
    A. R. Elcrat and K. G. Miller, Rearrangements in steady vortex flows with circulation, \textit{Proc. Amer. Math. Soc.}, 111(1991), 1051--1055.


    \bibitem{Elc2}
    A. R. Elcrat and K. G. Miller, Rearrangements in steady multiple vortex flows, \textit{Comm. Partial Differential Equations}, 20(1994), 1481--1490.

    \bibitem{Elc3}
    A. R. Elcrat and O. Neculoiu, Continuity of the profile function of a steady ideal vortex flow, Advances in geometric analysis and continuum mechanics (Stanford, CA, 1993), 74--80, Int. Press, Cambridge, MA, 1995.
	
    \bibitem{Gan}
    F. Gancedo, Existence for the $\alpha$-patch model and the QG sharp front in Sobolev spaces, \textit{Adv. Math.}, 217(2008), no. 6, 2569--2598.

    \bibitem{Gar1}
    C. Garc\'ia, T. Hmidi and J. Soler, Non uniform rotating vortices and periodic orbits for the two-dimensional Euler equations, \textit{Arch. Ration. Mech. Anal.}, 238(2020), no. 2, 929--1085.

    \bibitem{Gar2}
    C. Garc\'ia, Vortex patches choreography for active scalar equations,  Preprint arXiv:2010.07361v1.

    \bibitem{Go0}
    L. Godard-Cadillac, Smooth traveling-wave solutions to the inviscid surface quasi-geostrophic equations, Preprint arXiv:2010.09289.

	\bibitem{Go}
	L. Godard-Cadillac, P. Gravejat and D. Smets, Co-rotating vortices with $N$ fold symmetry for the inviscid surface quasi-geostrophic equation,
	arXiv: 2010.08194.


    \bibitem{Gom}
    J G\'omez-Serrano, J. Park, J. Shi and Y. Yao, Symmetry in stationary and uniformly-rotating solutions of active scalar equations, Preprint 	arXiv:1908.01722.

    \bibitem{Gra}
    P. Gravejat and D. Smets, Smooth travelling-wave solutions to the inviscid surface quasigeostrophic equation, \textit{Int. Math. Res. Not.}, (6)2019, 1744--1757.



	
    \bibitem{Has0}
    Z. Hassainia, N. Masmoudi Wheeler and H. Miles, Global bifurcation of rotating vortex patches, \textit{Comm. Pure Appl. Math.} 73(2020), no. 9, 1933--1980.


	\bibitem{Has}
	Z. Hassainia and T. Hmidi, On the V-states for the generalized quasi-geostrophic equations,
	\textit{Comm. Math. Phys.}, 337(2015), no. 1, 321--377.
	
    \bibitem{He}
    I. M. Held, R. T. Pierrehumbert, S. T. Garner and K. L. Swanson, Surface quasi-geostrophic dynamics, \textit{J. Fluid Mech.}, 282(1995), 1--20.

    \bibitem{Hmi}
    T. Hmidi, J.Mateu and J. Verdera, Boundary regularity of rotating vortex patches, \textit{Arch. Ration. Mech. Anal.}, 209(1)(2013), 171--208.

    \bibitem{Hmi1}
     T. Hmidi and J.Mateu, Degenerate bifurcation of the rotating patches, \textit{Adv. Math.}, 302(2016), 799--850.

    \bibitem{Hmi2}
     T. Hmidi and J.Mateu, Bifurcation of rotating patches from Kirchhoff vortices, \textit{Discret, Contin. Dyn. Syst.}, 36(10)(2016), 5401--5422.

	\bibitem{HM}
	T. Hmidi and J. Mateu, Existence of corotating and counter-rotating vortex pairs for active scalar equations,
	\textit{Comm. Math. Phys.}, 350(2017), 699--747.


    \bibitem{Kir}
    G. Kirchhoff, Vorlesungen uber mathematische Physik, Leipzig, 1874.


   \bibitem{Kis}
   A. Kiselev and F. Nazarov, A simple energy pump for the surface quasi-geostrophic equation,
  Nonlinear partial differential equations, volume 7 of Abel Symposia, Springer-Verlag, Berlin, Heidelberg, 2012, 175--179.
	
	\bibitem{Kis1}
	A. Kiselev, L. Ryzhik, Y. Yao and A. Zlato, Finite time singularity for the modified SQG patch equation,
	\textit{Ann. of Math.}, (2)184(2016), no. 3, 909--948.
	
	\bibitem{Kis2}
	A. Kiselev, Y. Yao and A. Zlato, Local regularity for the modified SQG patch equation,
	\textit{Comm. Pure Appl. Math.}, 70(2017), no. 7, 1253--1315.

    \bibitem{Kur}
    L. G. Kurakin and V. I. Yudovich, The stability of stationary rotation of a regular vortex polygon, \textit{Chaos}, 12(2002), no. 3, 574--595.
	
    \bibitem{Lamb}
    H. Lamb, Hydrodynamics, Cambridge University Press, Cambridge, 3rd ed., 1906.



    \bibitem{La}
    G. Lapeyre, Surface quasi-geostrophy, \textit{Fluids}, 2, 2017.

    \bibitem{Li0}
    D. Li, Existence theorems for the 2D quasi-geostrophic equation with plane wave initial conditions, \textit{Nonlinearity}, 22(2009), no. 7, 1639--1651.

	\bibitem{Lie}
	E. H. Lieb and M. Loss, Analysis,
	\textit{Graduate Studies in Mathematics, Vol. 14.}, American Mathematical Society, Providence, Second edition, RI (2001).
	
    \bibitem{Lin1}
    C. C. Lin, On the motion of vortices in two dimensions. I. Existence of the Kirchhoff-Routh function, \textit{Proc. Natl. Acad. Sci. USA}, 27(1941), 570--575.

    \bibitem{Lin2}
     C. C. Lin, On the motion of vortices in two dimensions. II. Some further investigations on the Kirchhoff-Routh function, \textit{Proc. Natl. Acad. Sci. USA}, 27(1941), 575--577.


	\bibitem{MB}
	A. J. Majda and A. L. Bertozzi, Vorticity and incompressible flow,
	\textit{Cambridge Texts in Applied Mathematics, Vol. 27.}, Cambridge University Press, 2002.
	
    \bibitem{Mar}
    F. Marchand, Existence and regularity of weak solutions to the quasi-geostrophic equations in the spaces $L^p$ or $\dot{H}^{-\frac{1}{2}}$, \textit{Comm. Math. Phys.}, 277(2008), no. 1, 45--67.

    \bibitem{March}
    C. Marchioro and M. Pulvirenti, Mathematical theory of incompressible nonviscous fluids, Springer-Verlag, New York, 1994.

    \bibitem{Mel}
     V. V. Meleshko and G. J. F. van Heijst, On Chaplygin's investigations of two-dimensional vortex structures in an inviscid fluid, \textit{J. Fluid Mech.}, 272(1994), 157--182.

    \bibitem{Nah}
    A.R. Nahmod, N. Pavlović, G. Staffilani and N. Totz, Global flows with invariant measures for the inviscid modified SQG equations, \textit{Stoch PDE: Anal Comp}, 6(2018), 184--210.

    \bibitem{New}
    P. K. Newton, The $N$-vortex problem, Analytical Techniques, Springer, New York, 2001.


    \bibitem{Nor}
    J. Norbury, Steady planar vortex pairs in an ideal fluid, \textit{Comm. Pure Appl. Math.}, 28(1975), 679--700.

    \bibitem{Poc}
     H. C. Pocklington, The configuration of a pair of equal and opposite hollow and straight vortices of finite cross-section, moving steadily through fluid, \textit{Proc. Camb. Phil. Soc.}, 8(1895), 178--187.


    \bibitem{Res}
    S. Resnick, Dynamical problems in non-linear advective partial differential equations. Ph.D. thesis, The University of Chicago, 1995.

    \bibitem{Rod}
     J. L. Rodrigo, On the evolution of sharp fronts for the quasi-geostrophic equation, \textit{Comm. Pure Appl. Math.}, 58(6)(2005), 821--866.

    \bibitem{Ros}
    M. Rosenzweig, Justification of the point vortex approximation for modified surface quasi-geostrophic equations, \textit{SIAM J. Math. Anal.}, 52(2020), no. 2, 1690--1728.



	\bibitem{SV}
	D. Smets and J. Van Schaftingen, Desingularization of vortices for the Euler equation,
	\textit{Arch. Rational Mech. Anal.}, 198(2010), 869--925.

    \bibitem{Tho}
    W. Thomson (Lord Kelvin), Maximum and minimum energy in vortex motion. In: Mathematical and Physical Papers, volume 4, Cambridge University Press, 1910, 172--183.

	
	\bibitem{T1}
	B. Turkington, On steady vortex flow in two dimensions. I, II,
	\textit{Comm. Partial Differential Equations}, 8(1983), 999--1030, 1031--1071.
	
	\bibitem{T2}
	B. Turkington, Corotating steady vortex flows with $N$-fold symmetry,
	\textit{Nonlinear Anal.}, 9(1985), no. 4, 351--369.

    \bibitem{Wan}
    Y. H. Wan and M. Pulvirenti, Nonlinear stability of circular vortex patches, \textit{Commun. Math. Phys.}, 99(1985), 435--450.

    \bibitem{Wu1}
    J. Wu, Quasi-geostrophic-type equations with initial data in Morrey spaces, \textit{Nonlinearity}, 10(1997), no. 6, 1409--1420.

    \bibitem{Wu2}
    J. Wu, Solutions of the 2D quasi-geostrophic equation in H$\ddot{\text{o}}$lder spaces, \textit{Nonlinear Anal.}, 62(2005), no. 4, 579--594.

    \bibitem{Yang}
    J. Yang, Existence and asymptotic behavior in planar vortex theory, \textit{Math. Models Methods Appl. Sci.}, 1(1991), 461--475.
	
    \bibitem{Yu}
     H. Yu, X. Zheng and Q. Jiu, Remarks on well-posedness of the generalized surface quasi-geostrophic equation, \textit{Arch. Ration. Mech. Anal.}, 232(1) (2019), 265--301.


	\bibitem{Yud}
	V. I. Yudovich, Non-stationnary flows of an ideal incompressible fluid,
	\textit{Zhurnal Vych Matematika}, 3(1963), 1032--1106.
	
	
	
	
	
	
	
	
	
	
\end{thebibliography}
\end{document}